\documentclass[12pt]{amsart}
\usepackage[utf8]{inputenc}
\usepackage{amsmath, amssymb, amsthm}
\usepackage{mathtools}
\usepackage{hyperref}
\usepackage{comment}
\usepackage{import}
\usepackage{enumitem}
\usepackage{url}
\usepackage[table,xcdraw, dvipsnames]{xcolor}
\usepackage{pdfpages}
\usepackage[a4paper,
            bindingoffset=0in,
            left=1.5cm,
            right=1.5cm,
            top=1.8cm,
            bottom=1.8cm,
            footskip=.25in]{geometry}

\usepackage{graphicx}
\usepackage{pinlabel}
\usepackage{tikz}
\usetikzlibrary{decorations.pathreplacing, positioning, calc}
\usepackage{caption}
\usepackage{subcaption}
\graphicspath{ {./image/} }
\usepackage{tikz-cd}
\tikzcdset{scale cd/.style={every label/.append style={scale=#1},
    cells={nodes={scale=#1}}}}

\newtheorem{thm}{Theorem}[section]
\newtheorem{theorem}[thm]{Theorem}
\newtheorem{constr}[thm]{Construction}

\newtheorem{lemma}[thm]{Lemma}
\newtheorem{prop}[thm]{Proposition}

\newtheorem{corollary}[thm]{Corollary}

\theoremstyle{definition}
\newtheorem{defn}[thm]{Definition}

\newtheorem*{theorem*}{Theorem}
\newtheorem*{problem*}{Problem}
\newtheorem*{corollarynonu}{Corollary}

\newtheorem{innercustomcor}{Corollary}
\newenvironment{customcor}[1]
  {\renewcommand\theinnercustomcor{#1}\innercustomcor}
  {\endinnercustomcor}

\def\R{\mathbb{R}}

\def\Z{\mathbb{Z}}

\def\res{\raise-.5ex\hbox{\ensuremath|}}

\makeatletter
\newcommand{\customlabel}[2]{%
   \protected@write \@auxout {}{\string \newlabel {#1}{{#2}{\thepage}{#2}{#1}{}} }%
   \hypertarget{#1}{#2}
}
\makeatother

\makeatletter
\def\@cite#1#2{{\normalfont[{#1\if@tempswa , #2\fi}]}}
\makeatother

\DeclareMathOperator{\Img}{\mathrm{im}}

\DeclareMathOperator{\Spin}{\mathrm{Spin}}
\DeclareMathOperator{\Arf}{\mathrm{Arf}}
\DeclareMathOperator{\Hom}{\mathrm{Hom}}
\DeclareMathOperator{\Map}{\mathrm{Map}}

\DeclareMathOperator{\Mod}{\mathrm{Mod}} 
\DeclareMathOperator{\Diff}{\mathrm{Diff}}

\usepackage[style=alphabetic, backref=true]{biblatex}
\AtBeginBibliography{\small}

\title[]{A surgery approach to abelian quotients of the level 2 congruence group and the Torelli group}
\author[]{Tudur Lewis}

\begin{document}
\maketitle
\begin{abstract}
    We provide algorithms for computing the Rochlin invariants of mod 2 homology spheres and mapping tori. This provides a unified framework for studying two families of maps: the Birman–Craggs maps of the Torelli group, and Sato’s maps of the level 2 congruence subgroup of the mapping class group. Our framework gives new, elementary proofs that both families of maps are homomorphisms, gives an explicit method for evaluating these maps on Dehn twists, and relates the two families when restricted to the Torelli group. It also gives a relation between an extension of the Birman--Craggs maps to the level 2 congruence subgroup, and Meyer's signature cocycle. Our methods involve 3--manifold techniques, and do not depend on results in 4--manifold theory as in the original constructions. 
\end{abstract}

\section{Introduction}
\subsection*{Background}
Let $\Sigma_{g,r}$ denote an oriented surface of genus $g$ with $r$ boundary components, and let $\Mod_{g,r} = \pi_0(\Diff^+(\Sigma_{g,r}, \partial \Sigma_{g,r}))$ denote the mapping class group of $\Sigma_{g,r}$. Let $\mathcal{I}_{g,r}$ denote the \textit{Torelli group}, that is, the kernel of the action of $\Mod_{g,r}$ on $H_1(\Sigma_{g,r}; \mathbb{Z})$. When $r=0$, we abbreviate $\Sigma_{g,0}$ to $\Sigma_g$.

The Torelli group arises in algebraic geometry as the fundamental group of Torelli space, the moduli space of genus $g$ Riemann surfaces $C$ with one boundary component and a symplectic basis of $H_1(C;\mathbb{Z})$. The abelianization of $\mathcal{I}_{g,1}$ is the first homology of Torelli space, and was calculated by Johnson in a series of papers. Johnson's work is of interest in the theory of $3$-- and $4$--manifolds. For example, Morita showed that the tools developed by Johnson have deep implications for the topology of homology $3$--spheres \autocite[Prop.2.3]{morita}. More recently, Lambert--Cole used Johnson's tools to study intersection forms of smooth $4$--manifolds, via trisections \autocite{cole}.

Johnson found that all torsion in the abelianization of $\mathcal{I}_{g,1}$ is characterised by the \textit{Birman--Craggs homomorphisms}, which are a family of maps $\mu_h: \mathcal{I}_{g,1} \rightarrow \mathbb{Z}/2$, indexed by a Heegaard embedding $h: \Sigma_g \rightarrow S^3$. These maps are constructed by cutting $S^3$ along the Heegaard surface $h(\Sigma_g)$, regluing the two handlebodies via an element of $\mathcal{I}_{g,1}$ to get a homology sphere, and then taking the Rochlin invariant of its unique spin structure \autocite{johnsonab},\autocite{birmancraggs}, \autocite{BCJpaper}. 

Let $M$ be an oriented Riemannian manifold, and let $P(M)$ denote the oriented orthonormal frame bundle of $M$. A \textit{spin structure} on $M$ is an element $\xi \in H^1(P(M);\Z/2) = \Hom(H_1(P(M)),\Z/2)$ that evaluates to $1$ on any homotopically non--trivial loop in a fibre of $P(M)$. The \textit{Rochlin invariant} of a spin $3$-manifold $M$ is defined as the signature modulo $16$ of any spin $4$-manifold spin bounding $M$; see Section \ref{sato_overview_section}. Let $\Spin(M)$ denote the set of spin structures on $M$, then $\Spin(M)$ is in bijection with $H^1(M;\Z/2)$.

Let $\Mod_{g,1}[L]$ denote the \textit{level $L$ congruence subgroup} of the mapping class group, that is, the kernel of the action of $\Mod_{g,1}$ on $H_1(\Sigma_{g,1} ; \mathbb{Z}/L)$, note that $\mathcal{I}_{g,1} < \Mod_{g,1}[L]$. The subgroup $\Mod_{g,1}[L]$ arises in algebraic geometry as the orbifold fundamental group of the moduli space of genus $g$ Riemann surfaces with one boundary component and a level $L$ structure. Farb posed the fundamental question of computing the abelianizations of these subgroups \autocite[Problem 5.23]{farbproblemlist}.

Our focus is on the \textit{level $2$ subgroup} $\Mod_{g,1}[2]$; this case often requires separate techniques. The level $2$ subgroup is of particular interest in the theory of $3$--manifolds due to its connections with rational homology spheres \autocite[Cor.1.1]{pitschriba}. To compute the abelianization of $\Mod_{g,1}[2]$, Sato defines a family of maps to abelian groups, using a construction that is similar to the Birman--Craggs maps \autocite[Part II]{sato}; the natural extension of the Birman--Craggs maps to $\Mod_{g,1}[2]$ is no longer a homomorphism \autocite[p.284]{birmancraggs}. Sato defines an analogous construction using mapping tori instead. Mapping tori have many spin structures, so Sato constructs a spin structure on the mapping torus $M_f$ of $[f] \in \Mod_{g,1}[2]$ that depends on a fixed spin structure $\sigma \in \Spin(\Sigma_g)$ for the fiber. Sato's maps $\beta_{\sigma, x} :\Mod_{g,1}[2] \rightarrow \mathbb{Z}/8$ are given by taking the Rochlin invariant of the spin mapping tori obtained from his construction; see Section \ref{sato_overview_section}.

Sato shows that the maps $\beta_{\sigma, x}$ give homomorphisms by using results on the signature of $4$-manifolds, such as Rochlin's theorem and Novikov additivity. A direct sum of a certain subfamily of the $\beta_{\sigma, x}$ was enough to compute the abelianization of $\Mod_{g,1}[2]$ \autocite[Lem.2.2, Prop.5.2 and 7.1]{sato}. His computation has applications in algebraic geometry; Putman built on Sato's work and computed the Picard groups of moduli spaces of curves with level structures in many cases; see \autocite{putmanduke} for a more algebraic approach to computing abelianizations of congruence subgroups.

\subsection*{Outline and main results}
Section \ref{sato_overview_section} overviews spin structures and Sato's definition of his maps $\beta_{\sigma, x}$. The main steps in our reconstruction of both maps is in Sections \ref{surgery_diagram_construction_section}, \ref{sato_new_def_section}, and \ref{birman_craggs_section}. Constructions \ref{constr_1}, \ref{spin_constr1}, and \ref{bc_extension_constr}, along with Theorem \ref{kirby_melvin_rochlin_formula} give an algorithm for computing Rochlin invariants of mapping tori and $\Z/2$--homology spheres.

In Section \ref{surgery_diagram_construction_section} we use framed links in $S^3$ and ribbon graphs to describe an algorithm that gives framed link diagrams of mapping tori and Heegaard splittings. This uses the $3$--manifold constructions in \autocite[Section 4]{ReshTur2}, and generalises the constructions in \autocite[Appendix]{kmtorusbundle} to higher genus. For all framed links obtained from our algorithm, fiber surfaces for the mapping tori (resp. Heegaard surfaces for the Heegaard splittings) lie in these surgery diagrams as the standard embedding of a surface into $S^3$. See Construction \ref{constr_1} for a summary of the algorithm.

In Section \ref{sato_new_def_section} we give a new definition of Sato's maps $\theta: \Spin(\Sigma_g) \rightarrow \Spin(M_f)$, where $[f] \in \Mod_{g,1}[2]$, and $M_f$ denotes the mapping torus of the map $f: \Sigma_g \rightarrow \Sigma_g$. Here, we fix a disk $D \subset \Sigma_g$, and think of representatives of elements in $\Mod_{g,1}[2]$ as diffeomorphisms of $\Sigma_g$ fixing $D$ pointwise. We begin with a framed link $L$ representing $S^1 \times \Sigma_g$. Using Construction \ref{constr_1}, we obtain a framed link $L_f$, containing $L$, representing $M_{f}$. Let $M_{L_f}$ denote the mapping torus obtained by Dehn surgery along $L_f$ in $S^3$. Then $M_{L_f}$ has a fixed embedding $\Sigma_g \hookrightarrow M_{L_f}$ representing a fiber surface, and a fixed embedding $S^1 \times D \hookrightarrow M_{L_f}$. For a spin structure $\sigma$ on $\Sigma_g$, we use Construction \ref{spin_constr1} to obtain spin structure $\theta_{L_f}(\sigma)$ on $M_{L_f}$. This spin structure is characterised by the fact that it restricts to $\sigma$ on the fiber surface, and restricts to a fixed spin structure on the embedding of $S^1 \times D$. This spin structure on $M_{L_f}$ gives an obstruction class $\omega_2(W_{L_f},s)$ that corresponds to a characteristic sublink $C$ of $L_f$, characterised by the condition $C \cdot L_i = L_i \cdot L_i \pmod{2}$ for all components $L_i$ of $L_f$; see Definition \ref{ffu} and Lemma \ref{char_sublink_bijection_spin}.

We give a new definition of Sato's maps $\beta_{\sigma, x}$ in terms of the characteristic sublinks obtained from our map $\theta_{L_f}$, using a formula for the Rochlin invariant found in \autocite[Appendix C.3]{kirbymelvin1}. Our definition involves the Arf invariant; the Arf invariant is a $\mathbb{Z}/2$-valued invariant of knots in $S^3$, which can be extended to an invariant of proper links in $S^3$. The main result is:
\begin{theorem*}
\label{mainresult}
For $g \geq 1$ let $\sigma \in \Spin(\Sigma_g)$, and let  $x \in H^1(\Sigma_g; \mathbb{Z}/2)$. Then Sato's maps $\beta_{\sigma, x} : \Mod_{g,1}[2] \rightarrow \mathbb{Z}/8$ can be evaluated as: 
\begin{equation*}
    \beta_{\sigma, x}(f) = (\theta_{L_f}(\sigma+x)\cdot \theta_{L_f}(\sigma +x)-\theta_{L_f}(\sigma) \cdot \theta_{L_f}(\sigma) + 8(\Arf(\theta_{L_f}(\sigma)) - \Arf(\theta_{L_f}(\sigma +x))))/2 \pmod 8.
\end{equation*}
Here $\theta_{L_f}(\sigma) \cdot \theta_{L_f}(\sigma)$ and $\Arf(\theta_{L_f}(\sigma))$ denote the total linking number and the Arf invariant of the characteristic sublink specified by $\theta_{L_f}(\sigma)$ in Construction \ref{spin_constr1}, and we use the affine action of $H^1(\Sigma_g; \mathbb{Z}/2)$ on the set of spin structures.
\end{theorem*}
Our main Theorem gives a direct mechanism for evaluating Sato's maps on any product of Dehn twists in $\Mod_{g,1}[2]$, and gives a new proof that Sato's maps are homomorphisms.
\subsection*{Applications to the Birman--Craggs maps and the Torelli group}
In Section \ref{birman_craggs_section}, we use Section \ref{surgery_diagram_construction_section} to give framed links for homology spheres; see Construction \ref{hom_sphere_constr}. We use a formula for the Rochlin invariant in \autocite[Appendix C.3]{kirbymelvin1} to give a framework for studying the Birman--Craggs maps. It is remarkable that the Birman--Craggs maps are homomorphisms \autocite[Thm.8]{birmancraggs}; we give a new proof of this fact in Theorem \ref{bcishom}. The idea of the proof is that gluing along a composition of diffeomorphisms translates to concatenation of tangle diagrams representing the Heegaard splittings. We then conclude Section \ref{birman_craggs_section} by calculating Sato's maps on bounding pairs and separating twists in Corollaries \ref{maincor2} and \ref{maincor3}. This relates Sato's maps to the Birman--Craggs maps using direct methods. We get:
\begin{corollarynonu}
Let $a,b$ be a pair of simple closed curves on $\Sigma_{g,1}$ that bound a subsurface. Let $\eta$ be the spin structure on $\Sigma_g$ with the characteristic sublink of $\theta_{L_{t_at_b^{-1}}}(\eta)$ containing none of the components from the link $L$ that represents $S^1 \times \Sigma_g $ (see Construction \ref{spin_constr1} and Section \ref{bp_eval_subsection}). If $\sigma = f^*(\eta)$ and $\sigma + x = h^*(\eta)$ for $[f],[h] \in \Mod_{g,1}$, then
\begin{equation*}
    \beta_{\sigma, x}(t_at_b^{-1}) = \mu_{\iota}(t_{f(a)}t_{f(b)}^{-1})-\mu_{\iota}(t_{h(a)}t_{h(b)}^{-1}) \pmod{2},
\end{equation*}
where $\mu_{\iota}$ denotes the Birman--Craggs map for the standard embedding $\iota:\Sigma_g \hookrightarrow S^3$. In particular, if $g \geq 3$, then we have $\beta_{\sigma,x} = \mu_{\iota \circ f}-\mu_{\iota \circ h} \pmod{2}$.
\end{corollarynonu}
In his survey on the Torelli group, Johnson asks if there is a definition of the Birman--Craggs maps that does not involve the implicit construction of a $4$--manifold \autocite[p.177]{johnsonsurvey}. Our definition uses Construction \ref{hom_sphere_constr}, and a formula computed from framed link diagrams of a $3$-manifold \autocite[Appendix C.3 and C.4]{kirbymelvin1}. To prove that this formula is well-defined, the fundamental theorem of Kirby calculus is used, and there is no dependence on Rochlin's theorem (see remark under \autocite[Cor.C.5]{kirbymelvin1}). Furthermore, there are proofs of the fundamental theorem of Kirby calculus that only use a presentation of the mapping class group \autocite{lukirbycalc}, \autocite{polyak}. We have given a definition of the Birman-Craggs maps that uses $3$--manifold topology and knot theory, which removes the logical dependence on Rochlin's theorem. One question, then, is whether it is possible to push the $3$--manifold techniques down a dimension and give an inherently $2$--dimensional description of the Birman-Craggs maps. This would give a group theoretic description of the Rochlin invariant, as Johnson pointed out.
\subsection*{The Birman--Craggs maps and Meyer's signature cocycle}
There is a natural extension of $\mu_h: \mathcal{I}_{g,1} \rightarrow \Z/2$ to a map $\mu_h: \Mod_{g,1}[2] \rightarrow \Z/16$; if we cut $S^3$ along the Heegaard surface $h(\Sigma_g)$, and reglue via an element $[f] \in \Mod_{g,1}[2]$, we get a $\Z/2$--homology sphere $S(f)$, so we can take the Rochlin invariant of its unique spin structure. This map is no longer a homomorphism, but our methods imply that the failure of these extensions from being homomorphisms is measured by Meyer's signature cocycle, restricted to $\Mod_{g,1}[2]$; Meyer's signature cocycle $\tau_g: \Mod_{g,1} \times \Mod_{g,1} \rightarrow \Z$ computes the signature of surface bundles with prescribed monodromy \cite{meyer}. In Section \ref{meyer_cocycle_section}, we show the following:
\begin{customcor}{4}
     For the standard embedding $\iota: \Sigma_g \rightarrow S^3$, there exists a well--defined map
    \begin{align*}
    \alpha_{\iota}: \Mod_{g,1}[2] \rightarrow \Z/16 \\
    f \mapsto \mathrm{Sign}(Y_f),
    \end{align*}
    where $Y_f$ is a cobordism between $S(f)$ and $M_f$, defined in Section \ref{meyer_cocycle_section}. Then we have $\tau_g \equiv \partial(\alpha_{\iota} + \mu_{\iota}) \pmod{16}$, where $\tau_g$ is Meyer's signature cocycle, and $\mu_{\iota}$ is the extension of the Birman--Craggs map described above.
\end{customcor}

Section \ref{meyer_cocycle_section} contains a formula for the extension $\mu_{\iota}: \Mod_{g,1}[2] \rightarrow \Z/16$ in terms of squares of Dehn twists; see Theorem \ref{bc_extension_formula}. This gives an algorithm for computing Rochlin invariants of $\Z/2$--homology spheres. The methods of Sections \ref{meyer_cocycle_section} and \ref{sato_new_def_section} combine to give an algorithm for evaluating $\tau_g$ on elements in $\Mod_{g,1}[2]$. It would be useful to find a closed formula for evaluating the map $\alpha_{\iota}: \Mod_{g,1}[2] \rightarrow \Z/16$ of Corollary \ref{meyer--birman--craggs}.

\subsection*{Acknowledgments}
The author thanks Tara Brendle for the guidance, encouragement, and advice she has given him, and thanks Benson Farb, Andy Wand, Brendan Owens, Vaibhav Gadre, and Dan Margalit for useful discussions and feedback. 

\section{Overview of Sato's homomorphisms} \label{sato_overview_section}

In this section, we review definitions of spin structures on manifolds, and Sato's construction of his homomorphisms.

We fix an embedded disc $D \subset \Sigma_g$, and think of $\Mod_{g,1}$ as the group of orientation--preserving diffeomorphisms fixing $D$ pointwise, modulo isotopies through maps of the same form. All homology groups are taken with coefficients in $\mathbb{Z}/2$ unless otherwise specified, and we use the same notation for a continuous map as its induced homomorphism on homology ($f=f_*$). Sato's idea is to take the mapping torus
$M_{f} = I \times \Sigma_g / (1,x) \sim (0,f(x))$ for $[f] \in \Mod_{g,1}$ and analyze the spin structures on $M_{f}$ induced by a given spin structure on $\Sigma_g$. 

\subsection*{Spin structures on manifolds}

Let $\pi:E \rightarrow V$ be a smooth oriented real vector bundle of rank $n \geq 2$ equipped with a metric and denote by $SO(n)  \overset{i}{\rightarrow}  P(E) \overset{p}{\rightarrow} V$ the oriented orthonormal frame bundle associated to this bundle. When the second Stiefel-Whitney class $\omega_2(E)$ vanishes, we have the short exact sequence
\begin{equation}
\label{ses1}
0 \rightarrow H_1(SO(n)) \overset{i}{\rightarrow} H_1(P(E)) \overset{p}{\rightarrow} H_1(V) \rightarrow 0.
\end{equation}
A \textit{spin structure} $\tau$ on $E$ is a homomorphism $\tau:H_1(V) \rightarrow H_1(P(E))$ such that $p \circ \tau = \mathrm{id}_{H_1(V)}$. We denote by $\Spin(E)$ the set of spin structures on $E$.

By the splitting lemma, the existence of $\tau$ as above is equivalent to the existence of a homomorphism $\tau': H_1(P(E)) \rightarrow H_1(\mathrm{SO}(n))$ such that $\tau' \circ i = \mathrm{id}_{H_1(\mathrm{SO}(n))}$. This gives a cohomology class $\tau \in H^1(P(E)) = \Hom(H_1(P(E)), \mathbb{Z}/2)$. This class $\tau$ can be evaluated on framed curves in $V$, and the condition $\tau \circ i = \mathrm{id}$ implies that $\tau$ evaluates to one on a trivial loop in $V$ with zero framing.

After identifying $H_1(SO(n))$ with $\mathbb{Z}/2$, there is a simply transitive action of $H^1(V) = \Hom(H_1(V), \mathbb{Z}/2)$ on $\Spin(E)$ given by taking a homomorphism $c: H_1(V) \rightarrow \mathbb{Z}/2$ and $\tau \in \Spin(E)$ and constructing another right splitting $\tau + i \circ c$. The number of spin structures for $ P(E) \overset{p}{\rightarrow} V$ is given by $|H^1(V)|$. We refer to a smooth manifold $M$ as spin if there exists a spin structure on the tangent bundle $TM$. Denote by $\Spin(M)$ the set of all spin structures on the tangent bundle $TM$ of $M$, whenever $M$ is a spin manifold.

The group $\Diff^+(M)$ acts on $\Spin(M)$ via pullback: for a diffeomorphism $g \in \Diff^+(M)$ and spin structure $\sigma: H_1(M) \rightarrow H_1(P(TM))$, we get the spin structure $g^*(\sigma) \coloneqq dg^{-1} \circ \sigma \circ g \in \Spin(M)$.

\subsection{Sato's construction.} \label{sato_constructions_subsection}

To define the homomorphisms $\beta_{\sigma, x}$, we must define a map $\theta : \Spin(\Sigma_g) \rightarrow \Spin(M_{f})$ for every given $[f] \in \Mod_{g,1}[2]$. We use the homotopy long exact sequence for the fibration $\Sigma_g \rightarrow M_{f} \rightarrow S^1$ and the fact that the abelianization functor is a right exact functor as well as a natural transformation between $\pi_1$ and $H_1$. If we combine this with the Wang exact sequence (see \autocite[Example 2.48]{Hatcher}), we get that the following sequence is exact:
\begin{equation}
\label{ses2}
0 \rightarrow H_1(\Sigma_g) \rightarrow H_1(M_{f}) \rightarrow H_1(S^1) \rightarrow 0,
\end{equation}
where the homomorphisms are induced by the inclusion and projection to $S^1$ respectively.
Since $f$ fixes $D \subset \Sigma_g$ pointwise we have an embedding $l:S^1 \times D \rightarrow M_{f}$, giving a right splitting of the short exact sequence (\ref{ses2}). This is equivalent to an isomorphism $H_1(M_{f}) \overset{h}{\rightarrow} H_1(S^1) \bigoplus H_1(\Sigma_g)$. A right splitting of the sequence (\ref{ses1}) for $V = M_{f}$ is obtained from $h$ in the following way:

Choose $p \in \mathrm{int}(D)$ and an arbitrary orthonormal frame $\{b_0, b_1 \}$ for $T_p(D)$. Pick a non-zero tangent vector field $v$ of $TS^1$ and denote by $v_t \in T_tS^1$ the value of $v$ at $t \in S^1$. For $S^1 \times D \subset M_{f}$ define the framing $\hat{l} : S^1 \rightarrow P(S^1 \times D)$ by 
\begin{equation*}
\hat{l}(t) = (v_t, b_0\cos(2\pi t)+ b_1\sin(2 \pi t), b_1\cos( 2 \pi t) - b_0\sin(2 \pi t)).
\end{equation*}
This frames the curve $S^1 \times \{p\}$ with a tangent vector field to the curve, and two transverse vector fields that rotate a total of $2 \pi$ in one traverse of the curve. This framing induces the homomorphism
\begin{equation}
\label{s1framing}
\hat{l}:H_1(S^1) \overset{\hat{l}}{\rightarrow} H_1(P(S^1 \times D)) \overset{\mathrm{inc}}{\rightarrow} H_1(P(M_{f})).
\end{equation}

For the $\Sigma_g$ factor, consider the smooth map $P((-\epsilon, \epsilon) \times \Sigma_g) \rightarrow P(M_{f})$ induced by the inclusion of a tubular neighbourhood $(-\epsilon, \epsilon) \times \Sigma_g$ of the fiber into $M_{f}$ for small $\epsilon >0$. Think of a spin structure $\sigma$ of $\Sigma_g$ as a right splitting of the sequence (\ref{ses1}) with $V = (-\epsilon, \epsilon) \times \Sigma_g$, and let $\overline{\sigma}:H_1(\Sigma_g) \rightarrow H_1(P(M_{f}))$ denote the following composition
\begin{equation}
\label{sframing}
\overline{\sigma}:H_1(\Sigma_g) \overset{\mathrm{inc}}{\rightarrow} H_1((-\epsilon, \epsilon) \times \Sigma_g) \overset{\sigma}{\rightarrow} H_1(P((-\epsilon, \epsilon) \times \Sigma_g)) \overset{\mathrm{inc}}{\rightarrow} H_1(P(M_{f})).
\end{equation}
Then construct a homomorphism $H_1(M_{f}) \rightarrow H_1(P(M_{f}))$ by combining (\ref{s1framing}) and (\ref{sframing}), to obtain a map $\theta: \Spin(\Sigma_g) \rightarrow \Spin(M_{f})$. 

In summary, the map $\theta$ inputs a spin structure $\sigma$ of $(-\epsilon, \epsilon) \times \Sigma_g$, and outputs the right splitting $(\hat{l} \oplus \overline{\sigma}) \circ h$ in the following commutative diagram.

\[\begin{tikzcd}[ampersand replacement=\&]
	0 \& {\Z/2} \& {H_1(P(M_f))} \& {H_1(M_f)} \& 0 \\
	\&\&\& {H_1(S^1) \oplus H_1(\Sigma_g)}
	\arrow[from=1-1, to=1-2]
	\arrow["i", from=1-2, to=1-3]
	\arrow["p", from=1-3, to=1-4]
	\arrow[from=1-4, to=1-5]
	\arrow["h", from=1-4, to=2-4]
	\arrow["{\hat{l} \oplus \overline{\sigma}}", from=2-4, to=1-3]
\end{tikzcd}\]

Now we describe the homomorphisms $\beta_{\sigma, x} :\Mod_{g,1}[2] \rightarrow \mathbb{Z}/8$ for $\sigma \in \Spin(\Sigma_g), x \in H^1(\Sigma_g)$. Rochlin's Theorem states that every spin $3$-manifold bounds a spin $4$-manifold. Fix a spin structure $\tau$ on $M_{f}$ and choose a compact spin $4$-manifold $V$ spin bounding $(M_{f}, \tau)$ and define the Rochlin invariant
\begin{equation*}
R(M_{f}, \tau) = \mathrm{Sign}(V) \pmod{16},
\end{equation*}
where $\mathrm{Sign}(V)$ is the signature of the intersection form of $V$.
This is well-defined by Novikov additivity, and Rochlin's result that a closed spin $4$-manifold has signature divisible by $16$.

\label{homom}
Let $\sigma \in \Spin(\Sigma_g)$ and $x \in H^1(\Sigma_g)$, define the map $\beta_{\sigma, x} :\Mod_{g,1}[2] \rightarrow \mathbb{Z}/8$ to be 
\begin{equation*}
\beta_{\sigma, x}([f]) = (R(M_{f} ,\theta(\sigma)) - R(M_{f}, \theta(\sigma + x)))/2 \pmod 8 .
\end{equation*}

Sato showed that these maps are homomorphisms and that they have image in $\mathbb{Z}/8$ \autocite[Lemmas 2.2 and 4.3]{sato}. He then examined the Brown invariant of a $Pin^-$ bordism class represented by a surface embedded in $M_{f}$ to arrive at a formula for the homomorphisms $\beta_{\sigma, x}$ on squares of Dehn twists \autocite[Prop. 5.2]{sato}. To describe the formula we need the following.

\subsection*{Spin structures and quadratic forms}
A \textit{symplectic quadratic form} is a map $q:H_1(\Sigma_g) \rightarrow \Z/2$ that satisfies $q(x+y) = q(x) + q(y) + x \cdot y$ for all $x,y \in H_1(\Sigma_g)$, where $x \cdot y$ denotes the pairing given by the intersection form.

\begin{theorem} \label{spin_strucures_quadratic_forms_bij}
    \cite[Theorems 3A,3B]{johnsonspin}
    There is a bijection $\sigma \mapsto q_{\sigma}$ between spin structures $\sigma \in \Spin(\Sigma_g)$ and symplectic quadratic forms $q_{\sigma}: H_1(\Sigma_g) \rightarrow \Z/2$. Let $H^1(\Sigma_g) = \Hom(H_1(\Sigma_g), \Z/2)$ act affinely on $\Spin(\Sigma_g)$ as above, let $f \in \Diff^+(\Sigma_g)$, and let $x \in H^1(\Sigma_g)$. Then $q_{\sigma + x} = q_{\sigma} + x$, and $q_{f^*(\sigma)} = f^*q_{\sigma}$.
\end{theorem}
\begin{proof}
We only sketch the bijection here: let $\sigma \in \Spin(\Sigma_g)$ and let $x \in H_1(\Sigma_g)$. Choose a simple closed curve $\alpha \subset \Sigma_g$ representing $x$. Let $N(\alpha)$ denote the normal bundle of $\alpha$ in $\Sigma_g$. Pick a unit tangent vector field $s: \alpha \rightarrow T(\alpha)$ and a nonzero section $X: \alpha \rightarrow N(\alpha)$. Viewing $\sigma \in \Spin(\Sigma_g)$ as a left splitting of the short exact sequence (\ref{ses1}) gives us a homomorphism $k_{\sigma}:H_1(P(\Sigma_g)) \rightarrow \mathbb{Z}/2$. Since $T(\Sigma_g) |_{\alpha} = N(\alpha) \bigoplus T(\alpha)$, we can define the associated quadratic form $q_{\sigma}$ to be
$$q_{\sigma}(x) = k_{\sigma}(X \oplus s) +1.$$
\end{proof}
Symplectic quadratic forms are determined by their values on a symplectic basis for $H_1(\Sigma_g)$, so we can specify an arbitrary spin structure by choosing the values of $q_{\sigma}$ on a fixed symplectic basis.

We need the following function to state Sato's formula for $\beta_{\sigma , x}(t_C^2)$.
For a homology class $z \in H_1(\Sigma_{g})$, define the map $i_z: H_1(\Sigma_{g}) \rightarrow \mathbb{Z}/8$ by 
\begin{equation*}
i_z(y) = 
   \left\{
\begin{array}{ll}
      1 , & z \cdot y =1 \pmod 2 \\
      0 , & z \cdot y =0 \pmod 2
\end{array} 
\right. 
\end{equation*}
where $\cdot$ denotes the intersection form on $H_1(\Sigma_{g})$.

\begin{prop}\autocite[Proposition 5.2]{sato} For a non-separating simple closed curve $C \subset \Sigma_g \setminus D$, we have \begin{equation*}
    \beta_{\sigma, x}(t_C^2) = (-1)^{q_{\sigma}(C)}i_{[C]}(PD(x)),
\end{equation*}
where $q_{\sigma}: H_1(\Sigma_g) \rightarrow \mathbb{Z}/2$ is the quadratic form associated to $\sigma \in \Spin(\Sigma_g)$ and $PD(x)$ is the Poincare dual of $x \in H^1(\Sigma_g)$.
\end{prop}

Since $\Mod_{g,1}[2]$ is generated by squares of Dehn twists about non separating simple closed curves this formula is enough to calculate the abelianization of $\Mod_{g,1}[2]$ (see \autocite[Proposition 2.1]{humphriesfactorisation}).
We give an alternative description of $\beta_{\sigma, x}$ that allows us to directly evaluate these maps. To find this formula we need to write mapping tori as surgery diagrams.

\section{Surgery diagrams and ribbon graphs} \label{surgery_diagram_construction_section}
In this section, we construct surgery diagrams of Heegaard splittings and mapping tori. We think of representatives of elements in $\Mod_{g,1}$ as orientation--preserving diffeomorphisms of $\Sigma_g$ fixing an embedded disk $D$ pointwise, so the $3$--manifolds we consider are closed. We define ribbon graphs of Heegaard splittings, and give a procedure for obtaining a surgery diagram of a mapping torus from this ribbon graph. The discussion on Heegaard splittings requires only $3$--manifold topology, while the procedure for obtaining mapping tori uses $4$--dimensional handlebodies, and dotted circle notation. For more information about the $3$-manifold constructions used here see \autocite[Appendix]{kmtorusbundle}, \autocite[Section 4]{ReshTur2} and \autocite[Section 2.2]{wright}. The terminology in the next paragraph is consistent with \autocite{gompfstipsicz}.

An \textit{$n$-dimensional $k$-handle} attached to a smooth manifold $M$ will be a copy of $D^k \times D^{n-k}$ attached to $\partial M$ via an embedding $(\partial D^k) \times D^{n-k} \rightarrow \partial M$. For a handlebody decomposition of a smooth $4$-manifold $M$, we assume there is one $0$-handle, and that any $4$-dimensional $1$-handles are attached to the boundary ($S^3$) of this $0$-handle; these $1$-handles can be pictured as two copies of $D^3$ in $S^3$ identified to each other via a reflection. Any $4$-dimensional $2$-handles $D^2 \times D^2$ attached to the manifold can be specified by drawing the attaching circle $(\partial D^2) \times \{0\}$ along with a framing of the normal bundle for this attaching circle in $\partial M$. There is a bijection between these framings and the integers, explained below. Call a collection of framed links in $S^3$ with embedded pairs of $D^3$ in $S^3$ a \textit{Kirby diagram} of $M$.

If $M$ has only $4$-dimensional $2$-handles attached to a $0$-handle, we call $M$ a $2$-handlebody. Every $2$-handle $D^2 \times D^2$ is attached along an embedding $\eta : (\partial D^2) \times D^2 \rightarrow S^3$; the $(\partial D^2) \times D^2$ part of the boundary of $D^2 \times D^2$ is then in the interior of the new manifold, and the $D^2 \times (\partial D^2)$ factor changes the boundary $3$-manifold. On the boundary it is equivalent to removing a tubular neighbourhood of the attaching circle and gluing in a solid torus $D^2 \times (\partial D^2)$ by sending the meridian curves $\partial D^2 \times \{pt\}$ to their images under the embedding $\eta$. This is called \textit{Dehn surgery}, and the Kirby diagram of $M$ is also a surgery diagram for $\partial M$. The Dehn-Lickorish theorem states that any closed orientable $3$-manifold can be described by such a surgery diagram \autocite{lickorishannals}.

\subsection{Framings of $2$-handles} \label{framing_2_handle_subsection}
Let $\widetilde{\varphi}: \partial D^2 \rightarrow \partial M$ be an embedding, and let $\nu(\partial D^2)$ be the normal bundle of $\widetilde{\varphi}$ in $\partial M$. Choose a framing $\{ s_1, s_2 \}$ of $\nu(\partial D^2)$, and a tubular neighbourhood $N: \nu(\partial D^2) \rightarrow \partial M$. We obtain a gluing map $\varphi: (\partial D^2) \times D^2 \rightarrow \partial M$ for a $2$-handle by setting 
$\varphi(x,a,b) = N(as_1(\widetilde{\varphi}(x))+bs_2(\widetilde{\varphi}(x)))$. The meridians $(\partial D^2) \times \{pt\}$ are glued to pushoffs of the attaching circle $\widetilde{\varphi}$ along these frames. Note that the \textit{core} $\{0\} \times \partial D^2$ of this added solid torus is sent to a meridian of the attaching circle $\widetilde{\varphi}$ in the Kirby diagram.

Suppose the $4$-dimensional $2$-handles are attached to $\partial D^4 = S^3$ in the notation above; there is a bijection between framings of a $4$-dimensional $2$-handle $D^2 \times D^2$ and $\pi_1(SO(2)) = \mathbb{Z}$, but this correspondence requires a choice of an arbitrary framing. Define the $0$-framing to be the non--zero transverse vector field to the attaching circle $\varphi(\partial D^2 \times \{0\}) = K$ in $S^3$ induced from the collar of any Seifert surface for $K$. For the pushoff $K'$ of $K$ in the direction of the $0$-framing, we have $lk(K, K') = 0$, and the bijection between framings and $\mathbb{Z}$ is realised by linking numbers of pushoffs. The vector field corresponding to $k \in \mathbb{Z}$ is given by a vector field which deviates from the $0$-framing by $k$ full twists (right handed twist is +1). Note that this framing integer is independent of the orientation chosen for the attaching circle $K$, since reversing $K$ also reverses the pushoff $K'$ in the direction of the vector field. For the embedding $\varphi: \partial D^2 \times D^2 \rightarrow S^3$ of the previous paragraph, its framing is given by $lk(\varphi(\partial D^2 \times 0), \varphi(\partial D^2 \times 1))$.

\subsection{Surgery diagrams for Heegaard splittings}  \label{heegaard_splitting_subsection}
To aid in the constructions below, we use the following definitions. A \textit{coupon} is an embedding of $I^2$ in the interior of the unit cube $I^3$. Fix a collection of coupons in $I^3$, and define a tangle to be an embedding of an oriented $1$-manifold in $I^3$, with its boundary contained in either $1/2 \times I \times \{ 0,1\}$ or the top or bottom edges $I \times \{0,1\}$ of coupons. Tangles are considered up to isotopy keeping the endpoints fixed. A framing of a tangle is an embedding of its normal bundle up to isotopy; this embedding is specified by a choice of parallel pushoff for the tangle, as above. We assume that the framing of any tangle connected to a coupon is parallel to the top or bottom edges of the coupon. A \textit{ribbon graph} is a collection of framed tangles and coupons in $I^3$, where tangles only intersect coupons at their endpoints. An example of a ribbon graph used often is given in Figure \ref{ribbongraph}. Refer to this ribbon graph as $\Delta_g$ and denote by $-\Delta_g$ its inversion.
\begin{figure}
    \centering
    \includegraphics[scale = 0.2]{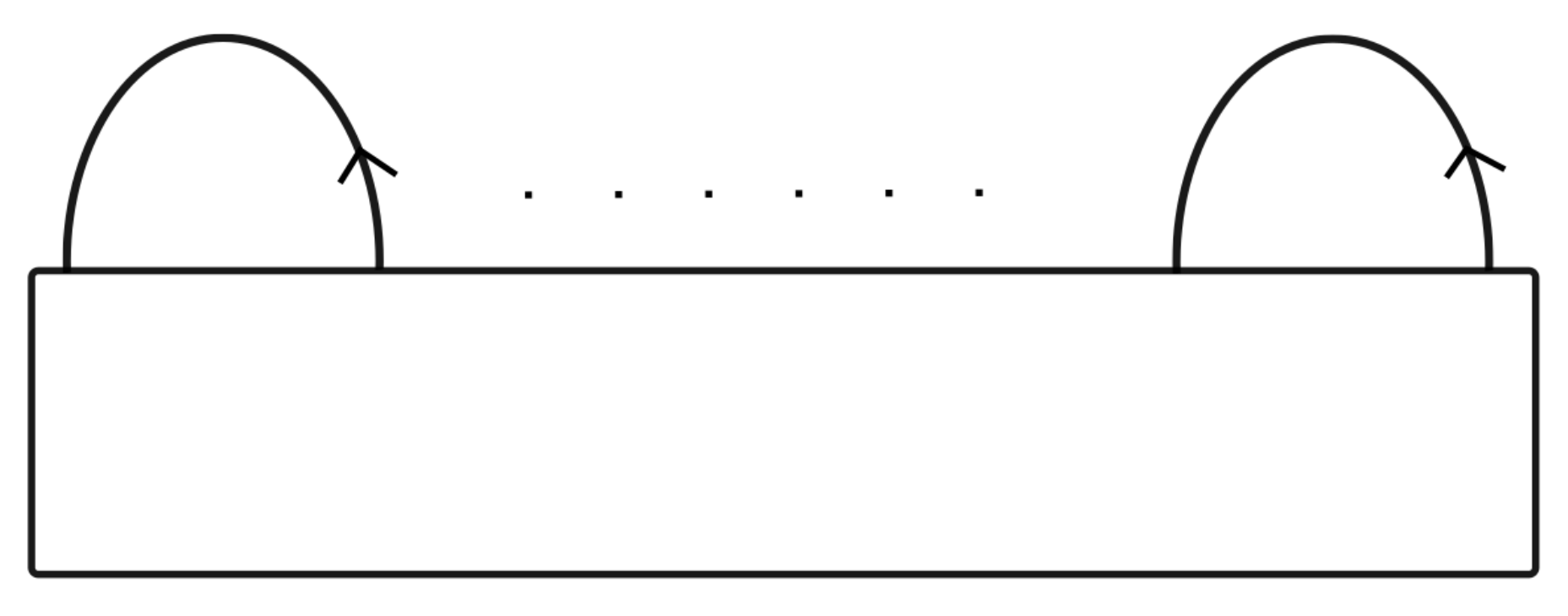}
    \caption{The ribbon graph $\Delta_g$, where all tangles are given the $0$-framing}
    \label{ribbongraph}
\end{figure}

Let $H_g$ be an oriented handlebody of genus $g$; we assume that $H_g$ has one $0$-handle $D^3$ and $g$ $3$-dimensional $1$-handles all attached to the $0$-handle. Let $-H_g$ denote the handlebody $H_g$ with opposite orientation. We fix our model of $\pm H_g$ to be a regular neighbourhood of $\pm \Delta_g$ in $S^3$, and we fix $\partial(H_g) = \Sigma_g$ as our model for a surface of genus $g$. Given an element $[f] \in \Mod_{g,1}$ we can form the following closed $3$-manifold
\begin{equation*}
    S_g(f) = H_g \cup_f -H_g = H_g \sqcup I \times \Sigma_g  \sqcup -H_g / \sim,
\end{equation*}
where $(0,x) \in \{ 0 \} \times \Sigma_g  \sim x \in \partial(H_g)$ and $(1,x) \in \{ 1 \} \times \Sigma_g  \sim f(x) \in \partial(-H_g)$. We refer to the manifold $S_g(f)$ as a \textit{Heegaard splitting} of genus $g$. When the genus is clear from context, we abbreviate $S_g(f)$ to $S(f)$. There is an embedding of $\Sigma_g$ in this manifold given by $\partial(H_g)$ that we refer to as a \textit{Heegaard surface}.

Let $L$ be a framed link in $S^3 -\pm \Delta_g$ such that surgery along $L$ produces $S(f)$ with $\pm \Delta_g$ embedded in $S^3 \setminus L$ as in Figure \ref{ribbongraph}. We get a ribbon graph $L \cup \pm \Delta_g$ for the manifold $S(f)$, viewed as a Kirby diagram for $S(f)$ with added data. To get a framed link $L$ representing $S(f)$ in this way, note that $\#_g S^1 \times S^2 = S(\mathrm{id})$ has a surgery diagram given by $g$ disjoint unknots with $0$--framing, and $H_g, -H_g$ are given by tubular neighbourhoods of the copies of $\pm \Delta_g$ in Figure \ref{identityhandlebody} .

\begin{figure}
    \centering
    \includegraphics[scale = 0.3]{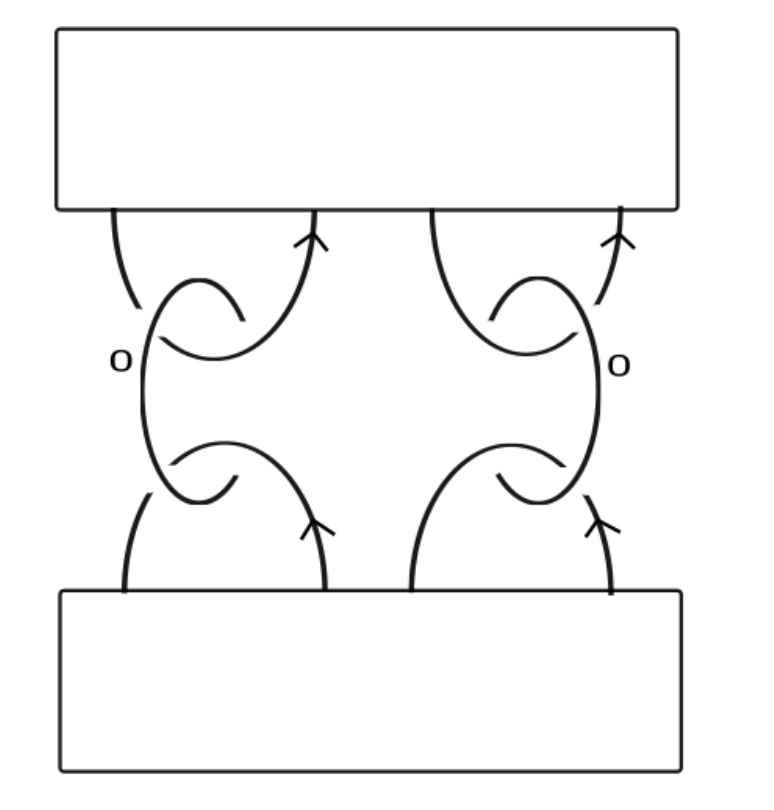}
    \caption{$\#_g S^1 \times S^2 = H_g \bigcup_{id} -H_g$ with $\pm \Delta_g$ embedded in them.}
    \label{identityhandlebody}
\end{figure}

Take a positive normal to $\Sigma_g$ in $S^3$ and take an embedding of $I \times \Sigma_g$ into $S^3$ specified by this framing. Think of $\{ 0 \} \times \Sigma_g$ as $\partial(H_g)$ and $\{ t \} \times \Sigma_g $ as a pushoff of $\Sigma_g$ in the direction of the positive normal, that points out the page. Denote the image of this embedding of $I \times \Sigma_g$ into $S^3$ by $\nu(\Sigma_g)$. We ensure $\nu(\Sigma_g)$ does not intersect any tangles in the diagram. Next, we modify the gluing map from $S(id)$ to $S(f)$ using $\nu(\Sigma_g)$. 

Let $c$ be a simple closed curve in $\Sigma_g \setminus D$. To obtain $S(t_c^{\pm1})$ from $S(\mathrm{id})$ using Dehn surgery, pick a fiber $\{\mathrm{pt}\} \times \Sigma_g := F$ in $\nu(\Sigma_g)$ and choose an embedding of $c$ in $F$, let $A$ be an annular neighbourhood of $c$ in $F$, and write $N$ for the solid torus obtained by thickening $A$ to one side of $F$ in $\nu(\Sigma_g)$. Now remove $N$ and reglue it by the map of Figure \ref{stalltwist}. Away from $N$ the fibering is the same, but as we pass across $N$ a Dehn twist about $c$ occurs, so we have cut open the fiber surface $F$, and reglued via $t_c^{\pm 1}$. This is equivalent to performing Dehn surgery along $c$, with framing specified by Figure \ref{stalltwist}. 

\begin{figure}
    \centering
    \includegraphics[scale=0.3]{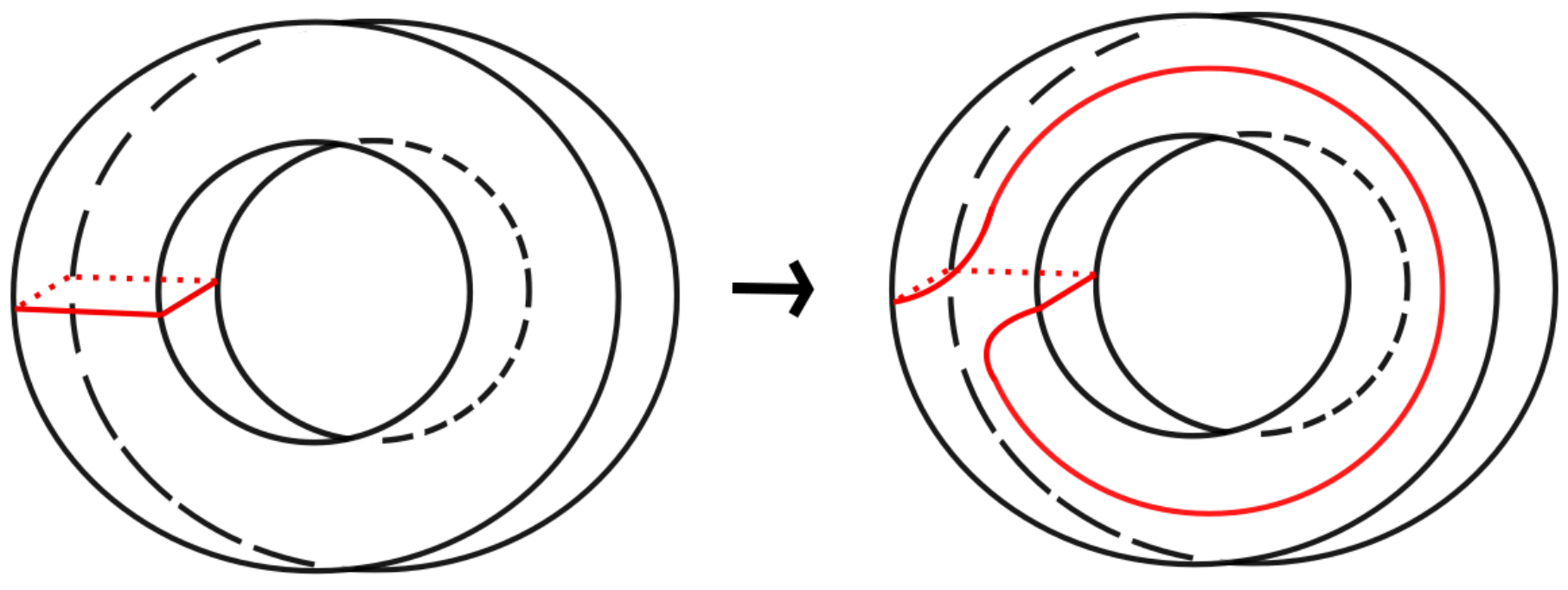}
    \caption{}
    \label{stalltwist}
\end{figure}

Next, suppose we have $f$ written as a product of Dehn twists $t_{c_n} \cdot \cdot \cdot t_{c_1}$. Since $\nu(\Sigma_g) = I \times \Sigma_g$, we pick $t_1 < t_2 < \cdot \cdot \cdot < t_n \in I$ and place the curve $c_i$ in the Kirby diagram at $\{ t_i \} \times \Sigma_g$. Dehn surgery along these curves using the map given by Figure \ref{stalltwist} gives $S(f)$. The framing and linking of these curves in the Kirby diagram is captured by Seifert's linking form, described below. An example of the tangle diagrams obtained from this method is given in Figure \ref{tangle1}.

\subsection*{Seifert pairing} 

Let $F$ be an oriented surface embedded in $S^3$; given a curve $a$ on $F$, let $a^+$ denote the pushoff of $a$ in the direction of the positive normal to $F$. We define Seifert's linking pairing 
\begin{equation*}
    \lambda : H_1(F; \mathbb{Z}) \times H_1(F; \mathbb{Z}) \rightarrow \mathbb{Z}
\end{equation*}
by the formula $\lambda(a,b) = lk(a, b^+)$. This gives a well-defined bilinear pairing that is an invariant of the ambient isotopy class of the embedding of $F$ into $S^3$; see \autocite[Chapter VII]{kauffmanknots}. We orient $S^3$ and $F$ so that the positive normal points out of the page, toward the reader. 

On the boundary, attaching a $4$-dimensional $2$-handle with framing $n$ along a knot $K$ is equivalent to removing a solid torus neighbourhood of $K$ and gluing a solid torus back in by sending a meridian to the pushoff $K'$ of $K$ in the direction of the transverse vector field of the framing.

We want to remove a torus neighbourhood $N$ of $c$, and reglue by the map of Figure \ref{stalltwist}. A meridian of the surgered torus is sent to the rightmost red curve of Figure \ref{stalltwist}, so this is equivalent to Dehn surgery along $c \subset F$ with framing $\lambda(c,c) \pm 1$, where $\pm1$ comes from the gluing being $t_c^{\pm1}$.

\begin{figure}
\centering
\begin{tikzpicture}
    \node (image) {\includegraphics[scale = 0.6]{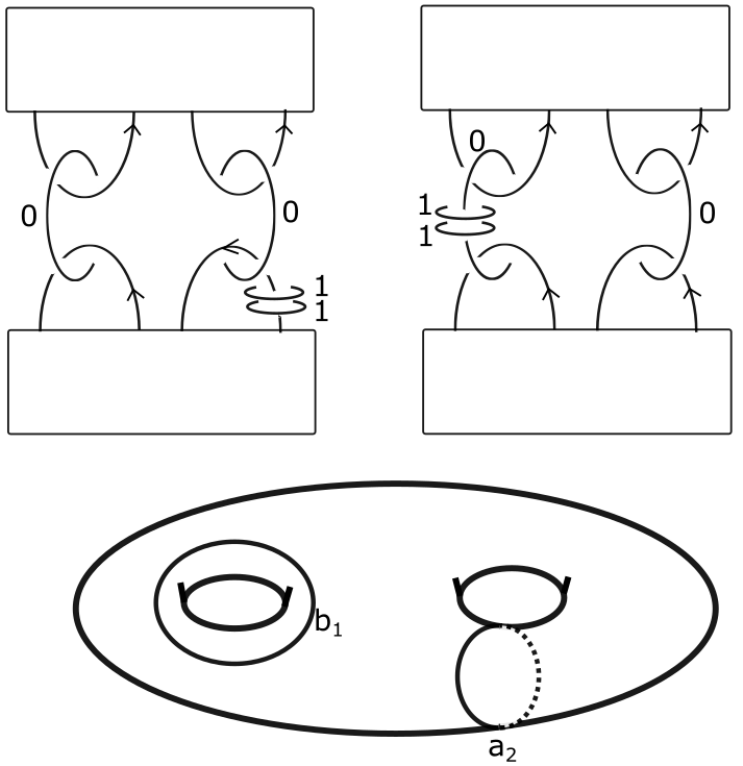}};
    \draw[decorate, decoration={brace, mirror, amplitude=5pt, raise=-2pt}, thick]
        ([yshift=-36pt]image.north west) -- ([yshift=130pt]image.south west)
        node[midway, xshift=-0.6cm, anchor=east] {$\Gamma_{t_{a_2}^2}$};
    
    \draw[decorate, decoration={brace, amplitude=5pt, raise=-3pt}, thick]
        ([yshift=-36pt]image.north east) -- ([yshift=130pt]image.south east)  
        node[midway, xshift=0.6cm, anchor=west] {$\Gamma_{t_{b_1}^2}$};
\end{tikzpicture}
\caption{Ribbon graphs for $S(t_{a_2}^2)$ (left) and $S(t_{b_1}^2)$(right).}
\label{tangle1}
\end{figure}

\subsection{Surgery diagrams of mapping tori} \label{mapping_tori_surgery_subsection}
Now we obtain framed links for mapping tori $M_f$ from the ribbon graphs $L \cup \pm \Delta_g$ of $S(f)$. Removing regular neighbourhoods $\nu (\pm \Delta_g)$ of $\pm \Delta_g$ in $S^3$ from our ribbon graph for $S(f)$ results in a manifold diffeomorphic to $I \times \Sigma_g$, cut along $\{ \mathrm{pt} \} \times \Sigma_g$, and reglued via $f$. After identifying the remaining boundary surfaces via the identity, we get $M_{f}$. Suppose we specify the manifold $S(f)$ by a ribbon graph $L \cup \pm \Delta_g$ in $S^3$ as above. Removing $\nu(\pm \Delta_g)$ and identifying the two boundary surfaces is equivalent to adding a copy of $D^1 \times H_g$ to $S(f)$ by gluing $-1 \times H_g$ to $\nu(-\Delta_g)$ and $1 \times H_g$ to $\nu(\Delta_g)$; the $\partial D^1 \times H_g$ part of the boundary of $D^1 \times H_g$ is then in the interior of the new manifold, and the $D^1 \times \partial H_g$ is a new part of the boundary, glued along $\partial D^1 \times \partial H_g$.

We assume that $H_g$ has one $0$-handle $D^3$, so we add one $4$-dimensional $1$-handle $D^1 \times D^3$ to the Kirby diagram, along with $g$ $4$-dimensional $2$-handles $D^1 \times D^1 \times D^2$ coming from the $g$ $3$-dimensional $1$-handles of $H_g$. These $4$-dimensional $2$-handles are attached with framing $0$ (draw the relevant portion $\partial(D^1 \times D^1) \times \{pt\}$ for $pt \in \partial D^2$ that is visible in the Kirby diagram). This gives us a Kirby diagram for a $4$-manifold with boundary $M_f$.

Think of the two $3$-balls $\partial{D}^1 \times D^3$ of the $4$-dimensional $1$-handle attached in Figures \ref{identityhandlebody}, \ref{tangle1} as tubular neighbourhoods of the coupons of $\pm \Delta_g$ in $S^3$. The $4$-dimensional $2$-handles are given by the tangles attached to the coupons (these $2$-handles run over the $1$-handle). The endpoints of the strands in the bottom coupon's tangle are matched with the endpoints of the strand in the top coupon's tangle directly above it. Denote the $4$-manifold given by this Kirby diagram by $X_{f}$.

\subsection*{Dotted circle notation} 

Now we have Kirby diagrams of $4$--manifolds with boundary $M_{f}$. Next, we describe Akbulut's dotted circle notation\autocite{akbulutdotcirc}.

If we smooth corners so that $D^2 \times D^2 = D^4$ we have a diffeomorphism $(D^2 \setminus \nu \{ 0 \}) \times D^2 = (S^1 \times D^1) \times D^2 = S^1 \times D^3$, where $\nu \{ 0 \}$ is a tubular neighbourhood of $0 \in D^2$. From this, we see that adding a $1$-handle to $D^4$ is the same as removing an open tubular neighbourhood of a properly embedded $2$-disc $\{ 0 \} \times D^2$, whose boundary $\{ 0 \} \times \partial D^2$ is visible in the Kirby diagram as an unknot in $S^3$; draw this unknot as a \textit{dotted circle} to indicate that it corresponds to a $1$--handle.

The $\nu \{ 0 \} \times \partial D^2$ part of the tubular neighbourhood  is visible in $\partial D^4 = S^3$ as a solid torus and the annulus $(D^2 \setminus \nu \{ 0 \}) \times \{ pt \}$ allows us to isotope a curve running through the $1$- handle to $\partial \nu \{ 0 \} \times pt$, which links once with the removed solid torus in $\partial D^4$. 

\begin{figure}
    \centering
    \includegraphics[scale=0.5]{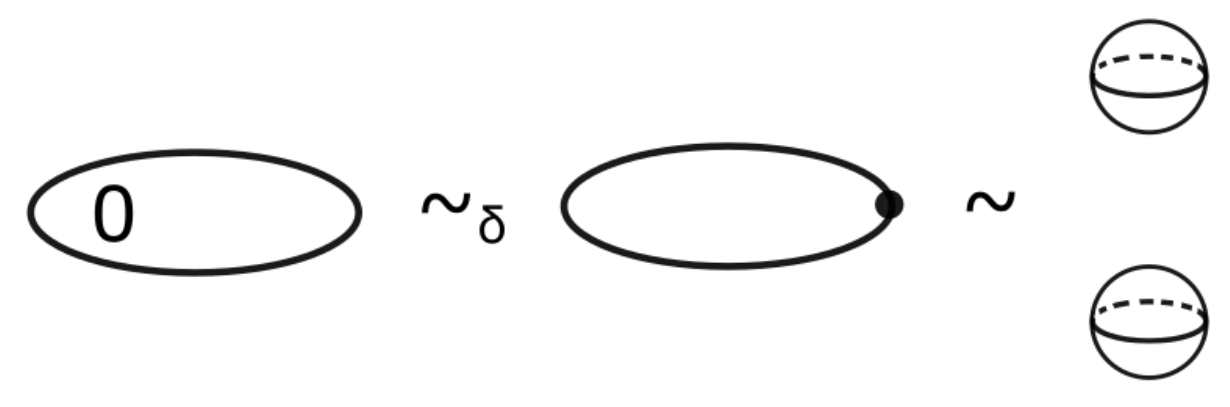}
    \caption{}
    \label{dot}
\end{figure}
Note that $\partial(S^1 \times D^3) = S^1 \times S^2 = \partial(D^2 \times S^2)$, where the Kirby diagram of $D^2 \times S^2$ is given by a $0$-framed unknot. Moving from the dotted circle to the $0$-framed unknot as in Figure \ref{dot} is done via surgery in the interior of the $4$-manifold, and the symbol $\sim_{\delta}$ denotes that there is a diffeomorphism between the two boundary $3$-manifolds.

Given a Kirby diagram with $1$-handles, we switch to dotted circle notation by isotoping the attaching circles of any $2$-handles so that they avoid the regions between the attaching balls of the $1$-handles. Then push the balls together and switch to dotted circle notation. After doing this, the curves running through the $1$-handle are linked with the dotted circle, and joined together by the gluing data of the $1$-handles. We use the convention of drawing dotted lines to indicate the paths taken to push any balls together; see \autocite[Section 5.4]{gompfstipsicz} for more details. Applying this operation to the Kirby diagram of the $4$-manifold $X_f$ obtained above, with $\partial X_f = M_{f}$, then changing the dotted circles to $0$-framed unknots, gives a framed link for $M_{f}$.

There is a natural way to switch to dotted circle notation using the ribbon graphs above; push the two balls given by the coupons together along a dotted line running to the right. An example is given by Figure \ref{surgery1}. In summary:

\begin{constr} \label{constr_1}
    Let $f: \Sigma_g \rightarrow \Sigma_g$ be a diffeomorphism with $[f] = t_{c_n}^{k_n} \cdots t_{c_1}^{k_1} \in \Mod_{g,1}$, where $k_i = \pm 1$ for all $i$. We have the following construction of framed link descriptions of $S(f)$ and $M_f$:
    \begin{enumerate}
        \item Start with a ribbon graph for $S(\mathrm{id}) = H_g \cup_{\mathrm{id}} -H_g$, as in Figure \ref{identityhandlebody}.
        \item Pick a collar, $I \times F$, of $F = \partial(\nu_{S^3}(\Delta_g))$ in $S^3$, and pick $t_1<\cdots < t_n \in I$. Then place $c_i$ in $\{ t_i \} \times F$ with framing $\lambda(c_i,c_i) + k_i$. Here $\lambda: H_1(F;\Z) \times H_1(F;\Z) \rightarrow \Z$ denotes the Seifert pairing of $F \subset S^3$. The closed components of this diagram give a framed link for $S(f)$, and the $\pm H_g$ are thought of as regular neighbourhoods of $\pm \Delta_g$ in $S^3$. 
        \item To go from $S(f)$ to $M_f$, think of the coupons of $\pm \Delta_g$ as $4$--dimensional $1$--handles, and the tangles of $\pm \Delta_g$ as $4$--dimensional $2$--handles. Then change to dotted circle notation as in Figure \ref{surgery1}, and replace the dotted circle by a $0$--framed unknot.
    \end{enumerate}
\end{constr}

After switching to dotted circle notation in construction \ref{constr_1} (3), the surface $F$ gets punctured, and can be visualised in the following way: choose a disk $D \subset S^3$, with boundary the dotted circle component that intersects the components of the framed link transversely in pairs of punctures. Take away small open disks in $D$ around the punctures, and replace with annuli that run along the components of the link intersecting $D$, to obtain a punctured fiber as in Figure \ref{vis_fibers}. The remaining part of the fiber, which is a disk, is in the surgered torus obtained from $0$--surgery on the dotted circle component. The complement of the dotted circle in $S^3$ fibers into disks, and the operation above allows us to see the other fibers of the mapping torus.

\begin{figure} 
\centering
\includegraphics[scale = 0.4]{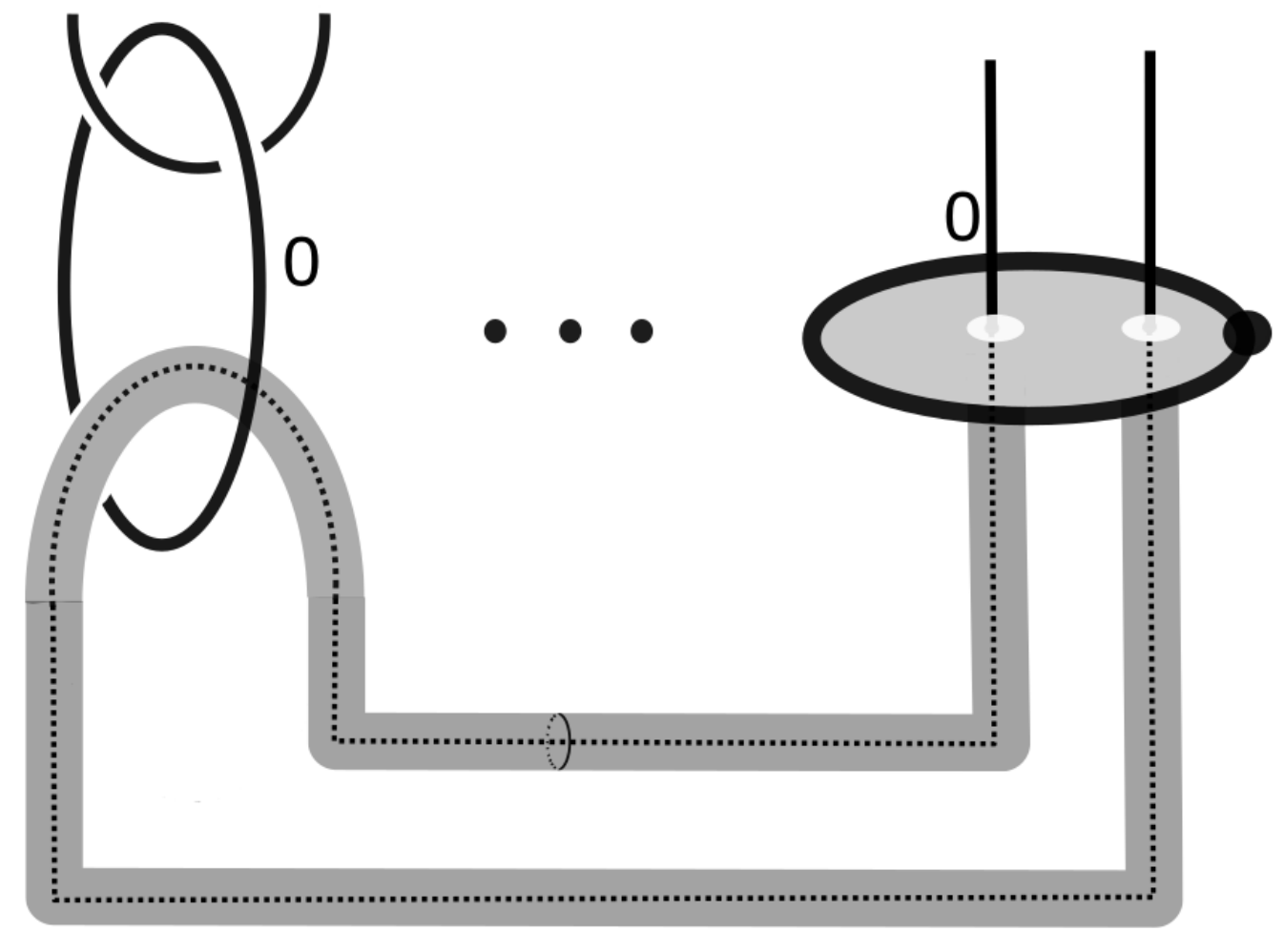}
\caption{A fiber in the surgery diagram for $M_f$ obtained from Construction \ref{constr_1}.}
\label{vis_fibers}
\end{figure}

\begin{figure}
    \centering
    \includegraphics[scale=0.6]{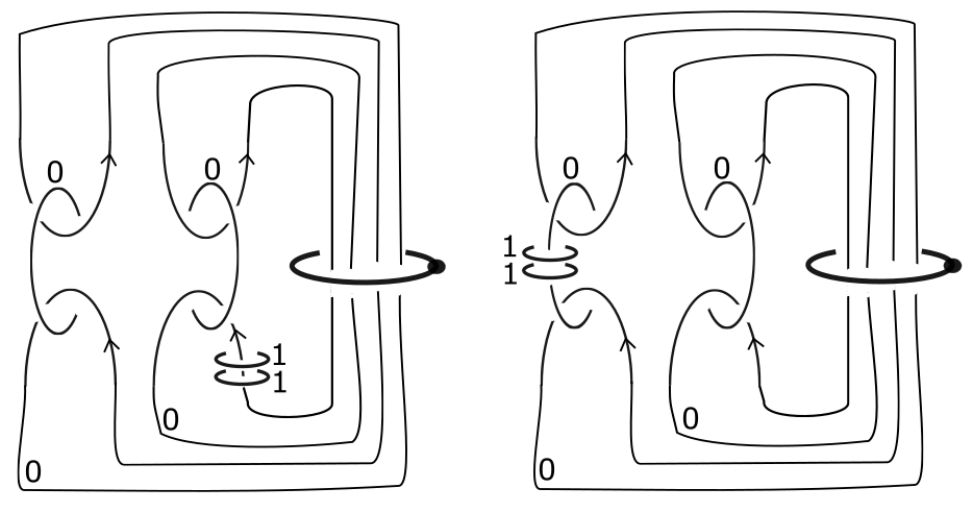}
    \caption{Surgery diagrams for $M_{t_{a_2}^2}$ and $M_{t_{b_1}^2}$ obtained from Figure \ref{tangle1}}
    \label{surgery1}
\end{figure}

\subsection{Surgery diagrams from composing tangles} \label{composing_tangles_subsection}

Surgery diagrams for Heegaard splittings and mapping tori can also be obtained by concatenating tangle diagrams; this is used later on for studying the Birman--Craggs maps.

The ribbon graphs obtained as in Figure \ref{tangle1} give surgery diagrams for $S(f) \coloneqq H_g \bigcup_f -H_g$ if we surger along the closed components of the tangle diagram. The handlebodies $\pm H_g$ are viewed as tubular neighbourhoods in $S^3$ of the $\pm \Delta_g$. This breaks $S(f)$ into three pieces, $\nu(\pm \Delta_g)$ and a copy of $I \times \Sigma_g$ cut open along a $\{ \mathrm{pt} \} \times \Sigma_g$ and reglued via $f$. If we remove tubular neighbourhoods in $S^3$ of the $\pm \Delta_g$ we get tangle diagrams for the mapping cylinder $C(f) = I \times \Sigma_g \bigsqcup \Sigma_g / \sim$, obtained by gluing $(1,x) \in \{ 1 \} \times \Sigma_g$ to $f(x) \in \Sigma_g$. Given such a ribbon graph of $S(f)$ corresponding to $C(f)$, we denote by $\Gamma_f$ the tangle diagram obtained by deleting the top and bottom coupons. The following useful result is due to Reshetikhin--Turaev \autocite[Lemma 4.4]{ReshTur2}, we include a proof of it here.

\begin{lemma}
\label{composingtangles}
Let $f,h: \Sigma_g \rightarrow \Sigma_g$ be diffeomorphisms, and let  $\Gamma_f, \Gamma_h$ be two tangle diagrams as above representing $C(f), C(h)$. Then the composition $\Gamma_f \circ \Gamma_h$ obtained by stacking $\Gamma_f$ on top of $\Gamma_h$ and then putting coupons on the ends gives a ribbon graph for $S(f \circ h)$. Here, the $g$ additional unknotted components obtained from stacking are given the $0$--framing.
\end{lemma}
\begin{proof}
    For $\Gamma_f \circ \Gamma_h$, consider the $g$ $0$--framed unknots $K_1,..,K_g$ obtained by gluing the top boundary tangles of $\Gamma_h$ to the bottom boundary tangles of $\Gamma_f$; see Figure \ref{tangle2} for an example with $g=2$. Each of the $K_i$ transversally hits a plane $\R^2 \times \frac{1}{2} \subset \R^3$, along which $\Gamma_h$ is glued to $\Gamma_f$. Complete this plane into a $2$--sphere $S^2 = \R^2 \times \frac{1}{2} \cup \{ \infty\} \subset S^3$, and take a cylinder, $S^2 \times [0, \epsilon]$, over this $2$--sphere, such that each of the $K_i$ meet this cylinder in two vertical segments $\{\mathrm{pt}\} \times [0,\epsilon]$.

    Now, surger $S^3$ along all the closed components of $\Gamma_f \circ \Gamma_h$, using the given framings, to get a closed $3$--manifold $\overline{M}$. In doing so, we cut out regular neighbourhoods $U_1,..,U_g$ of $K_1,..,K_g$, and glue in $g$ solid tori $W_1,...,W_g$; since the $K_i$ all have zero framing, we glue a meridian of $W_i$ to a zero framed longitude of $K_i$ in $U_i$. Let $U_i'$ be a regular neighbourhood of $K_i$ with $U_i \subset U_i'$, then
    $$N := (((S^2 \times [0, \epsilon]) \cup \bigcup_i U_i') - \bigcup_i U_i) \cup \bigcup_i W_i \subset \overline{M}$$
    is identified with $[0,1] \times \partial(\nu(\Delta_g))$; compare with Figure \ref{identityhandlebody}, the complement of the $\pm H_g = \nu(\pm \Delta_g)$ in $H_g \cup_{\mathrm{id}}-H_g$ is a copy of $I \times \partial(\nu(\Delta_g))$, and can be identified with $N$ by construction. 
    
    After composing the tangle diagrams and putting coupons on the two opposite ends, the resulting copies of $\pm \Delta_g$ in our ribbon graph have regular neighbourhoods in $S^3$ corresponding to $\pm H_g$. Hence, $\overline{M} \setminus \nu(\pm \Delta_g)$ can be identified with $([0,\frac{1}{3}] \times F) \bigcup_h ([\frac{1}{3}, \frac{2}{3}] \times F) \bigcup_f ([\frac{2}{3},1] \times F)$, where $F = \partial(\nu(\Delta_g))$, and $([0,\frac{1}{3}] \times F) \bigcup_h ([\frac{1}{3}, \frac{2}{3}] \times F)$ denotes that $(\frac{1}{3}, x) \in [0, \frac{1}{3}] \times F$ has been glued to $(\frac{1}{3}, h(x)) \in [\frac{1}{3}, \frac{2}{3}] \times F$, and similarly for the other union. This implies $\overline{M} \setminus \nu( \pm \Delta_g)$ is diffeomorphic to $C(f \circ h)$, so the ribbon graph obtained is a surgery diagram for $S(f \circ h)$.
\end{proof}

For an example of the constructions in the proof of Lemma \ref{composingtangles}, the tangle diagram for $C(t_{a_2}^2t_{b_1}^2)$ given in Figure \ref{tangle2} is obtained using the tangle diagrams for $C(t_{a_2}^2)$ and $C(t_{b_1}^2)$ given in Figure \ref{tangle1}. To summarise, we outline the following inductive construction: 

\begin{figure}[h]
\begin{tikzpicture}
    \centering
    \node (image) {\includegraphics[scale=0.5]{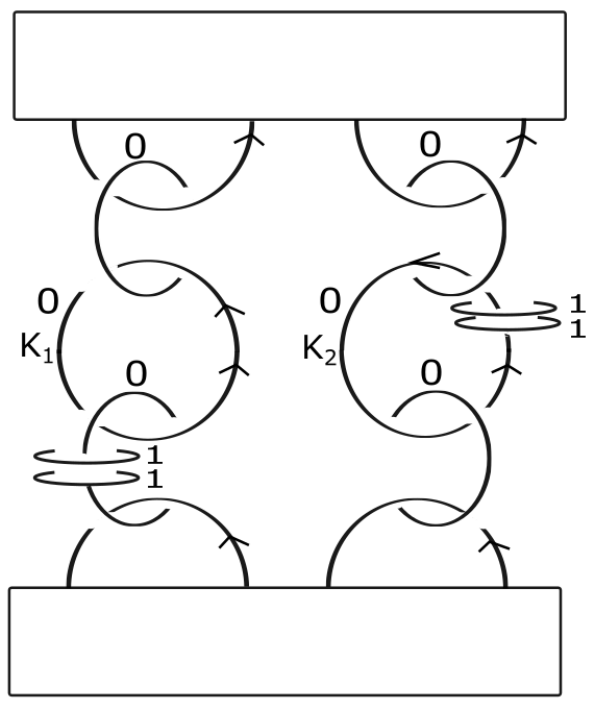}};
    \draw[decorate, decoration={brace, amplitude=10pt, raise=4pt}, thick]
        ([yshift=-35pt]image.north east) -- ([yshift=-92pt]image.north east)  
        node[midway, xshift=0.6cm, anchor=west] {$\Gamma_{t_{a_2}^2}$};
    \draw[decorate, decoration={brace, amplitude=10pt, raise=4pt}, thick]
        ([yshift=-92pt]image.north east) -- ([yshift=-150pt]image.north east)  
        node[midway, xshift=0.6cm, anchor=west] {$\Gamma_{t_{b_1}^2}$};
\end{tikzpicture}    
\caption{Surgery diagram for $S(t_{a_2}^2t_{b_1}^2)$ obtained by concatenating the tangles of Figure \ref{tangle1}}
\label{tangle2}
\end{figure}

\begin{constr} \label{constr_2}
    Let $f= f_n \cdots f_1$ be a diffeomorphism, where each $f_i$ is a composition of Dehn twists (or their inverses). Then, by concatenating tangle diagrams, we have the following inductive construction of framed link diagrams for $S(f)$ and $M_f$.
    \begin{enumerate}
        \item Use part $(2)$ of Construction \ref{constr_1} to obtain ribbon graphs of the $S(f_i)$, for $i=1,..,n$. Let $\Gamma_{f_i}$ denote the tangle diagrams obtained by deleting any coupons, as above.
        \item The composition $\Gamma_{f_n} \circ \cdots \circ  \Gamma_{f_1}$, obtained using Lemma \ref{composingtangles} inductively, with coupons added to the top and bottom, gives a ribbon graph for $S(f_n\cdots f_1)$; surgering along the closed components of $\Gamma_{f_n} \circ \cdots \circ  \Gamma_{f_1}$ gives a framed link for $S(f_n\cdots f_1)$.
        \item We obtain a framed link for $M_{f_n \cdots f_1}$ by viewing the coupons as $4$--dimensional $1$--handles, and the tangles (nonclosed components) as $4$--dimensional $2$--handles, then changing to dotted circle notation as in Figure \ref{surgery1}, and changing the dotted circle into a zero framed unknot.
    \end{enumerate}
\end{constr}

\section{Rewriting Sato's maps} \label{sato_new_def_section}

In this section, we define Construction \ref{spin_constr1}, that outputs a framed link diagram of a mapping torus, along with a characteristic sublink (defined below). This construction plays the role of Sato's map $\theta$ in Section \ref{sato_overview_section}. Construction \ref{spin_constr1} and Theorem \ref{kirby_melvin_rochlin_formula} provide an algorithm to evaluate the Rochlin invariant of spin mapping tori.

\subsection{Characteristic sublinks}

For a framed link $L$ in $S^3$, write $M_L$ for the $3$--manifold obtained via Dehn surgery on $L$. There is a $2$--handlebody $W_L$ with Kirby diagram $L$, and $\partial W_L = M_L$. The linking matrix of $L$ is also the matrix of the intersection form of $W_L$ with respect to a basis obtained from the components $L_i$ of $L$; use $\cdot$ to denote this intersection form.
\begin{defn}
\label{ffu}
Let $L$ be an oriented framed link with components $L_1,..,L_n$, and let $C$ be a sublink of $L$. Define $(w_i)_{i=1}^n \in (\Z/2)^n$ by $w_i = 1$ if $L_i$ is in $C$, and $w_i = 0$ else. Then $C$ is characteristic if
\begin{equation*}
    p_iw_i + \sum_{j \neq i}lk(L_i,L_j)w_j \equiv p_i \pmod{2},
\end{equation*}
for all $1 \leq i \leq n$. Here $p_i$ denotes the integer specifying the framing of $L_i$. We abbreviate these conditions to $C \cdot L_i \equiv L_i \cdot L_i \pmod{2}$.
\end{defn}
\begin{lemma} \label{char_sublink_bijection_spin}
\cite[Lemma C.1]{kirbymelvin1}, \cite[Prop. 5.7.11]{gompfstipsicz}
There is a natural bijection between spin structures on $M_L$ and characteristic sublinks of $L$.
\end{lemma}
\begin{proof}
The correspondence is given by taking a spin structure $s$ of $M_L$, and defining $C$ to be the union of all components $L_i$ of $L$ such that the spin structure does not extend over the $2$--handle in $W_L$ attached to $L_i$.
\end{proof}

\subsection{Rewriting Sato's constructions}

Let $L_f$ be a framed link for the mapping torus of $f$, obtained from Construction \ref{constr_1}, where $f$ is a product of squares of Dehn twists, bounding pairs, or separating twists. By Figure \ref{vis_fibers}, we see a punctured fiber $F$ for the mapping torus $M_{L_f}$, where $F$ lies in $S^3 \setminus L_f$ as the standard embedding of a surface into $S^3$. Let $L$ be the framed link obtained from Construction \ref{constr_1} ($3$) with $f = \mathrm{id}$, then $M_L \cong S^1 \times \Sigma_g$, and $L$ is a sublink of $L_f$. The framed link $L_f$ is obtained from $L$ by placing the curves appearing in the factorisation of $f$ in pushoffs of the fiber surface, and framing them using the Seifert pairing of $F \subset S^3$.

In Section \ref{sato_constructions_subsection}, for a spin structure $\sigma \in \Spin(\Sigma_g)$, we obtained a spin structure $\theta(\sigma)$ as a right splitting of the short exact sequence
\begin{equation}
    \label{framebundleses}
    0 \rightarrow H_1(SO(3)) \overset{i}{\rightarrow} H_1(P(M_{L_f})) \overset{p}{\rightarrow} H_1(M_{L_f}) \rightarrow 0.
\end{equation}
Under the splitting of $H_1(M_{L_f})$ obtained from the exact sequence (\ref{ses2}), it is given by $\hat{l} \oplus \overline{\sigma}$, where $\hat{l}$ is obtained from the fixed embedding of $S^1 \times D \hookrightarrow M_f$ in the notation of Section \ref{sato_constructions_subsection}. This construction gives a spin structure on $M_f$ that restricts to $\sigma \in \Spin(\Sigma_g)$ on a fiber, and restricts to a fixed spin structure on $S^1 \times D \subset M_f$.

For the framed link $L_f$ obtained from Construction \ref{constr_1}, we have a fixed embedding of a punctured fiber $F \subset S^3 \setminus L_f$, and a fixed embedding of $S^1 \times D \hookrightarrow S^3 \setminus L_f$ identified with a neighbourhood $\alpha$ of a meridian of the dotted circle; see Figure \ref{vis_fibers}. Let $B= \{a_i,b_i\}_{i=1}^g$ denote the standard basis for $H_1(F;\Z)$ as in Figure \ref{spbasis}. When placed in the fiber, the curves in $B$ go to meridians of the components of $L$, the sublink corresponding to $S^1 \times \Sigma_g$. We can frame these curves using a spin structure $\sigma \in \Spin(\Sigma_g)$ and a positive normal to the fiber $F$.

\begin{figure}
    \centering
    \includegraphics[scale=0.5]{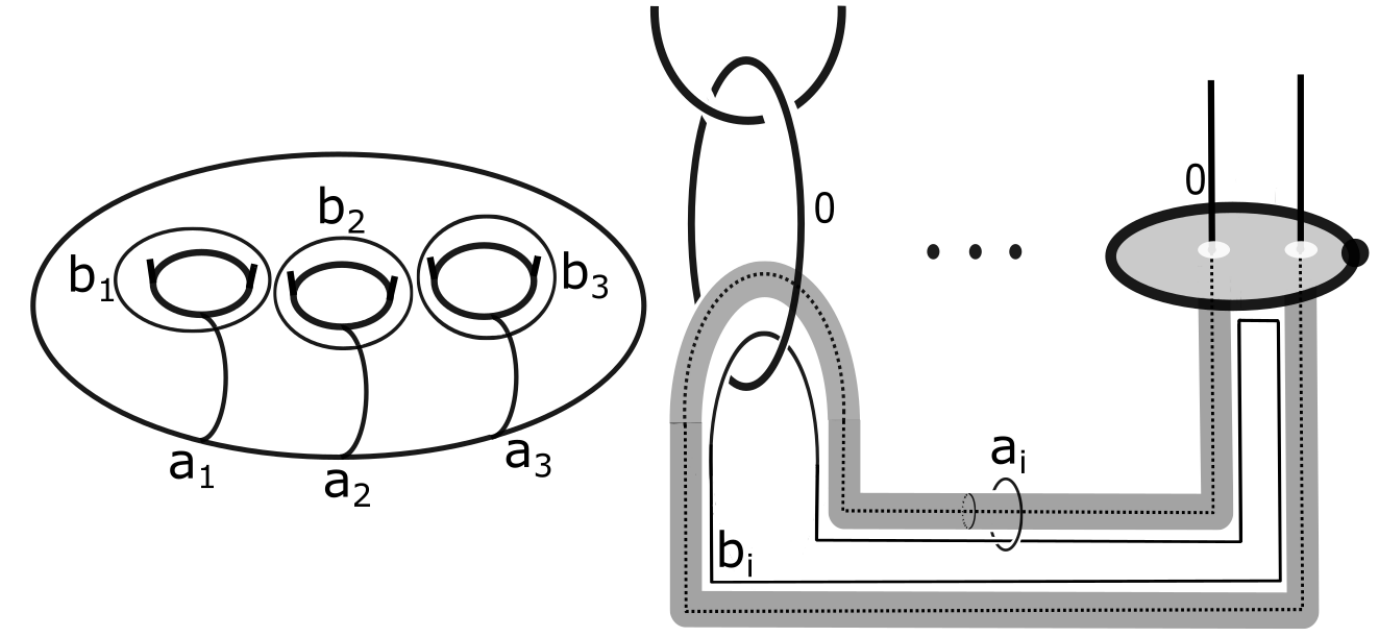}
    \caption{}
    \label{spbasis}
\end{figure}

To specify $\sigma \in \Spin(\Sigma_g)$, we use Theorem \ref{spin_strucures_quadratic_forms_bij}; spin structures on $\Sigma_g$ are in bijection with symplectic quadratic forms. Let $\sigma \in \Spin(\Sigma_g)$ have associated quadratic form $q_{\sigma}:H_1(\Sigma_g) \rightarrow \Z/2$. For an embedded circle $K \subset \Sigma_g$, the restriction of $\sigma$ to $K$ defines a spin structure on $T\Sigma_g|_K$, since $H^1(S^1) = \Z/2$, there are only two types up to homotopy. One is the bounding spin structure induced by the product framing on $D^2$. For a spin structure $\sigma \in \Spin(\Sigma_g)$, and an embedded curve $K \subset \Sigma_g$, we have that $q_{\sigma}([K]) = 0$ if and only if the spin structure $\sigma$ restricted to $K$ is spin bounding \cite[Remark 2.3]{blanchetmasbaumspin}.

\begin{lemma} \label{spin_bounding_implies_not_in}
    Let $\mu_j \in B \subset F$ be a meridian of a component $L_j$ of $L$ that doesn't come from a dotted circle, and let $\sigma \in \Spin(\Sigma_g)$. If $q_{\sigma}([\mu_j])=0$, then the spin structure on $M_{L_f}$ extends over the $2$--handle in $W_{L_f}$ attached along $L_j$.
\end{lemma}
\begin{proof}
    Since $q_{\sigma}([\mu_j])=0$, the spin structure $\sigma$ restricted to $\mu_j$ is spin bounding. As noted in Section \ref{framing_2_handle_subsection}, after surgery along $L_j$, the meridian $\mu_j$ can be isotoped to the core $\{0\} \times \partial D^2$ of the $2$--handle $h = D^2 \times D^2$ attached along $L_j$. The spin structure restricted to $h$ has to be the product framing, which agrees with the spin bounding structure on $\mu_j$.
\end{proof}

Using Lemmas \ref{char_sublink_bijection_spin} and \ref{spin_bounding_implies_not_in}, we define the characteristic sublink of $L_f$ in the following way: we set the components $L_i$ of $L$, with meridians $\mu_i \in B$ in the fiber $F$, satisfying $q_{\sigma}(\mu_i) =1$ to be in the characteristic sublink of $L_f$. The dotted circle component of $L_f$ is never in the characteristic sublink, this corresponds to fixing the spin structure on the embedding of $S^1 \times D$ given by the meridian $\alpha$. The conditions of definition \ref{ffu} determine whether the additional components of $L_f$ are in the characteristic sublink, by the following.

\begin{prop} \label{mon_comp_determined}
    Let $f: \Sigma_g \rightarrow \Sigma_g$ be a product of squares of Dehn twists, bounding pairs, or separating twists. Let $L_f$ be the framed link for $M_f$ obtained from Construction \ref{constr_1}, and let $L$ be the sublink corresponding to $S^1 \times \Sigma_g$. Then once we have defined which components of $L$ are in the characteristic sublink, the conditions of definition \ref{ffu} uniquely determine if the additional components of $L_f$ are in the characteristic sublink.
\end{prop}
\begin{proof}
    We prove the case where $f$ is a product of squares. Let $L_f$ have components $L_1,...,L_n$, and suppose that the components $L_i, L_{i+1}$ correspond to a $t_c^{\pm2}$ factor of the monodromy $f$, via Construction \ref{constr_1} (2). Then $p_i = p_{i+1} = \lambda(c,c) \pm 1$, where $\lambda$ is the Seifert pairing for the punctured fiber surface $F$ that lies in $S^3 \setminus L_f$. Applying the relations of definition \ref{ffu}, we have
    \begin{equation*}
        p_iw_i +(p_i-1)w_{i+1} + \sum_{j \neq i,i+1}lk(L_i,L_j)w_j \equiv p_i \pmod{2}.
    \end{equation*}
    The components $L_i$ and $L_{i+1}$ have the same linking number up to signs with all other components, since they are isotopic to each other in the fiber that lies ambiently in $S^3$, so applying definition \ref{ffu} to $L_{i+1}$ we get
    \begin{equation*}
        p_iw_{i+1} + (p_i-1)w_i + \sum_{j \neq i,i+1}lk(L_{i+1}, L_j)w_j \equiv p_iw_{i+1} + (p_i-1)w_i + \sum_{j \neq i,i+1}lk(L_i, L_j)w_j \equiv p_{i+1} \equiv p_i \pmod{2}.
    \end{equation*}
    But then 
    \begin{equation*}
        w_i \equiv w_{i+1} \equiv p_i -\sum_{j \neq i,i+1}lk(L_i,L_j)w_j \pmod{2}.
    \end{equation*}
    The result follows since the other components corresponding to the monodromy come in pairs, so they contribute zero to the right hand side of the last equation. The general case with bounding pair and separating twist factors involved follows from a similar argument, since the linking number of the components of $L_f$ is determined by the Seifert pairing on homology.
\end{proof}

For the symplectic quadratic form $q_{\sigma}: H_1(\Sigma_g) \rightarrow \Z/2$ associated to $\sigma \in \Spin(\Sigma_g)$, the affine action of $x \in \Hom(H_1(\Sigma_g),\Z/2) = H^1(\Sigma_g)$ on $\sigma$ satisfies the following identity: $q_{\sigma +x} = q_{\sigma}+x$ \cite[Theorem 3A]{johnsonspin}. For example, if $x \in H^1(\Sigma_g)$ satisfies $x([c]) =1$ for $c \in B$, then $q_{\sigma +x}([c]) = q_{\sigma}([c]) + 1$, so we switch whether the component of $L$ with meridian $c$ is in the characteristic sublink.

In summary, we have the following construction; recall any element of $\Mod_{g,1}[2]$ can be written as a product of squares of Dehn twists \cite[Prop.2.1]{humphriesfactorisation}.

\begin{constr} \label{spin_constr1}
Let $[f] \in \Mod_{g,1}[2]$ be a mapping class, written as a product of squares of Dehn twists.
\begin{enumerate}
    \item We start with a framed link $L$ representing $S^1 \times \Sigma_g$, obtained by applying Construction \ref{constr_1} $(3)$ to $f = \mathrm{id}$. A punctured fiber surface $F$ for $\{ \mathrm{pt} \} \times \Sigma_g$ is visible in $S^3 \setminus L$ as in Figure \ref{vis_fibers}. Construction \ref{constr_1} gives a framed link $L_f$ representing $M_f$, obtained by placing the curves appearing in the factorisation of $f$ in pushoffs of the fiber surface $F$ in $S^3 \setminus L$, and framing these curves using the Seifert pairing; $L_f$ has $L$ as a sublink.
    \item For $\sigma \in \Spin(\Sigma_g)$ with associated quadratic form $q_{\sigma}:H_1(\Sigma_g) \rightarrow \Z/2$, define a map $\Spin(\Sigma_g) \rightarrow \{$Characteristic sublinks of $L\}$ in the following way: a component $L_j$ of $L$ with meridian $\mu_j \in B$ in the fiber $F$, is in the characteristic sublink if and only if $q_{\sigma}(\mu_j) = 1$. The dotted circle component of $L$ is never in the characteristic sublink.
    \item Denote our map $\theta: \Spin(\Sigma_g) \rightarrow \{ \mathrm{Characteristic} \; \mathrm{sublinks} \; \mathrm{of} \; L_f \}$ by $\theta_{L_f}$ to indicate the dependence on the monodromy $f$, and our chosen framed link $L$ for $S^1 \times \Sigma_g$. Then $\theta_{L_f}(\sigma)$ is defined by declaring which components of $L$ are in the characteristic sublink as in ($2$), and the relations of definition \ref{ffu} determines if the additional components of $L_f$ are in the characteristic sublink, by Proposition \ref{mon_comp_determined}.
\end{enumerate}
\end{constr}

For two different factorizations of $[f] \in \Mod_{g,1}[2]$ into products of squares of Dehn twists, the spin mapping tori obtained from Construction \ref{spin_constr1} are spin diffeomorphic.

\begin{theorem} \label{well_defined_sato}
The map that sends $[f] \in \Mod_{g,1}[2]$ to the spin diffeomorphism class of $(L_f, \theta_{L_f}(\sigma))$, where $(L_f, \theta_{L_f}(\sigma))$ is obtained via Construction \ref{spin_constr1}, is well--defined.
\end{theorem}
\begin{proof} Suppose, for example, that $f = t_{c_n}^2\cdots t_{c_1}^2, h = t_{d_m}^2 \cdots t_{d_1}^2$, and that $[f]=[h]$ in the mapping class group. In the construction above, the monodromy is modified via Dehn surgery in a single tubular neighbourhood of the punctured fiber surface $F$ obtained from $\partial(\nu(\Delta_g))$ after switching to dotted circle notation (see Figure \ref{vis_fibers}). Identify this neighbourhood with $I \times F$, and set $N(f)$ to be the manifold obtained from $I \times F$ by modifying the monodromy by Dehn surgery along the curves $c_i \subset t_i \times F$ as in Construction \ref{constr_1} (2). Define $N(h)$ similarly.

We isotope the pair of surgery components corresponding to the $t_{c_i}^2$ part of the gluing to lie on the same pushoff of $F$, then $N(f) = ([0,\frac{1}{n+1}] \times F) \cup_{t_{c_1}^2} ([\frac{1}{n+1}, \frac{2}{n+1}]\times F) \cup_{t_{c_2}^2} \cdots \cup_{t_{c_n}^2}([\frac{n}{n+1},1] \times F)$. Here, $([\frac{k-1}{n+1}, \frac{k}{n+1}] \times F) \cup_{t_{c_k}^2} ([\frac{k}{n+1}, \frac{k+1}{n+1}] \times F)$ means that $(\frac{k}{n+1},x) \in [\frac{k-1}{n+1}, \frac{k}{n+1}] \times F$ has been glued to $(\frac{k}{n+1}, t_{c_k}^2(x)) \in [\frac{k}{n+1}, \frac{k+1}{n+1}] \times F$. To see this, note that performing Dehn surgery along $c_i \subset \{t_i\} \times F$, using the framing $\lambda(c_i,c_i) \pm 1$ coming from the Seifert pairing of $F \subset S^3$, is equivalent to cutting along $\{t_i\} \times F$, and regluing via $t_{c_i}^{\pm 1}$, see Figure \ref{stalltwist}.

We have a map $\psi: N(f) \rightarrow C(f) := I \times F \sqcup F / (1,x) \sim f(x)$, defined on the decomposition of $N(f)$ above, as $\psi|_{[0,\frac{1}{n+1}]\times F} = \mathrm{id}\times \mathrm{id}$, and $\psi|_{[\frac{k}{n+1}, \frac{k+1}{n+1}] \times F} = \mathrm{id}\times (t_{c_1}^{-2} \cdots t_{c_k}^{-2})$, for $k>0$. This map has inverse $\psi^{-1}: C(f) \rightarrow N(f)$, where $\psi^{-1}|_{[\frac{k}{n+1}, \frac{k+1}{n+1}]\times F} = \mathrm{id} \times (t_{c_k}^2 \cdots t_{c_1}^2)$ for $k>0$, and $\psi^{-1}|_{[0,\frac{1}{n+1}]\times F} = \mathrm{id}\times \mathrm{id}$ on the $I \times F$ part of $C(f)$, and $\psi^{-1}(x) = (1,x)$, for $x \in F \subset C(f)$.

Since $[f] = [h]$, there is a map $H:I \times F \rightarrow I \times F$ given as $H(s,x) = (s, H_s(x))$, where $H_s$ is an isotopy with $H_0 = \mathrm{id}_F$, and $H_1 = f^{-1}h$. This defines a map $\phi: C(h) \rightarrow C(f)$, by $\phi(s,x) = H(s,x)$ for $(s,x) \in I \times F$, and $\phi(x) = x$, for $x \in F$. The map $G: I \times F \rightarrow I \times F$ given by $G(s,x) 
= (s, H_s^{-1}(x))$ induces an inverse map to $\phi$.

The maps of the previous two paragraphs descend to maps between mapping tori; the $3$--manifold specified by $L_f$ via Construction \ref{spin_constr1} can be identified with $N(f) / (1,x) \sim (0,x)$, and the abstract mapping torus $M_f$ can be identified with $C(f) / (1,x) \sim (0,f(x))$. We have $\psi(1,x) = (1,f^{-1}(x)) \sim (0,f \circ f^{-1}(x)) =(0,x)= \psi(0,x)$. The inverse of $\psi$ also induces a diffeomorphism between the mapping tori by a similar argument. The $\phi^{\pm 1}$ also descend to diffeomorphisms of mapping tori, since $\phi(1,x) = (1, f^{-1}h(x)) \sim (0, ff^{-1}h(x)) = (0,h(x)) = \phi(0,h(x))$ when $\phi$ is projected to the mapping torus using the identifications above.

Composing these diffeomorphisms gives a diffeomorphism between the mapping tori specified by $L_f$ and $L_h$. Locally, the total derivative of this diffeomorphism is of the form $\mathrm{id}_{\R} \oplus de$, for some $e \in \mathrm{Diff}(F)$ which preserves the spin structure restricted to $F$, and fixes pointwise a neighbourhood of the meridian of the dotted circle component. By comparing the action of this diffeomorphism on the spin structures obtained using the exact sequence (\ref{framebundleses}), we see that the spin mapping tori $(L_f, \theta_{L_f}(\sigma))$ and $(L_h, \theta_{L_h}(\sigma))$ obtained from Construction \ref{spin_constr1} are spin diffeomorphic, for fixed $\sigma \in \Spin(\Sigma_g)$.
\end{proof}
The proof of Theorem \ref{well_defined_sato} also works for the case where $f$ is a product of squares, bounding pairs, or separating twists.

\subsection{Sato's homomorphisms}

Denote by $\Map(H_1(\Sigma_g), \Z/8)$ the free $\Z/8$--module consisting of all functions $H_1(\Sigma_g) \rightarrow \Z/8$, and recall $H_1(\Sigma_g)= H_1(\Sigma_g;\Z/2)$. Define
$$\beta_{\sigma}: \Mod_{g,1}[2] \rightarrow \Map(H_1(\Sigma_g),\mathbb{Z}/8)$$ by $\beta_{\sigma}([f])(x) = \beta_{\sigma, PD(x)}([f])$. In this section, we prove that the maps $\beta_{\sigma}$ are homomorphisms using a formula for computing the Rochlin invariant from a surgery diagram with a labelled characteristic sublink. The proof sheds light on why it is a homomorphism on the subgroup generated by squares of Dehn twists, which coincides with $\Mod_{g,1}[2]$ \cite[Prop.2.1]{humphriesfactorisation}.

We use Construction \ref{spin_constr1}, and think of $\theta_{L_f}$ as a map from $\Spin(\Sigma_g)$ to framed links with characteristic sublink $(L_{f},\theta_{L_f}(\sigma))$, such that the corresponding spin manifold is spin diffeomorphic to $(M_{f}, \theta(\sigma)$). We use the following formula for the Rochlin invariant of $(L_{f},\theta_{L_f}(\sigma))$, found by Kirby--Melvin.

\begin{theorem} \label{kirby_melvin_rochlin_formula} \cite[(C.3), Theorem C.4, Corollary C.5]{kirbymelvin1}
    Let $(L,C)$ denote a framed link $L$, with a characteristic sublink $C$ as in definition \ref{ffu}. Suppose $L$ is oriented, then the Rochlin invariant of the corresponding spin manifold is given by
    \begin{equation*}
        R(L,C) = \Lambda_L - C \cdot C + 8 \Arf (C) \pmod{16},
    \end{equation*}
    where $\Lambda_L$ is the signature of the linking matrix of $L$, $C \cdot C$ is the sum of all entries in the linking matrix of $C$, and $\Arf(C) \in \Z/2$ is the Arf invariant of $C$, defined below. The formula above gives a spin diffeomorphism invariant for spin $3$--manifolds, and in particular, the formula above is independent of the choice of orientations for the components of $L$.
\end{theorem}

\subsection*{The Arf invariant of a proper link} 

We call a link $L$ proper if $lk(K, L - K)$ is even for every component $K$ of $L$. Let $F$ be an oriented Seifert surface for $L$, and let $i:L \rightarrow \partial F$ denote the inclusion. Then $\Img(i_*)$ is the radical of the intersection form on $H_1(F)$ and the Seifert self--linking form 
\begin{align*}
\lambda : H_1(F)\rightarrow \mathbb{Z}/2 \\
[a] \mapsto \mathrm{lk}(a,a^+)
\end{align*}
satisfies $\lambda|_{\Img(i_*)} = 0$ if and only if $L$ is a proper link. We can define $\Arf(L)$ to be the Arf invariant of the form on $H_1(F)/\Img(i_*)$ induced by $\lambda$.

For a proper link $L$, we can produce a knot $K$ by band connecting together all the components of $L$, as long as the bands respect the orientations chosen for the components of $L$. It is a fact that $\Arf(L) = \Arf(K)$; see \autocite{robertelloarf} or \autocite{hostearf} for more details.

\subsection*{Another definition of Sato's maps} 

Using Theorem \ref{kirby_melvin_rochlin_formula}, we write the Rochlin invariant of the spin mapping torus $(L_f, \theta_{L_f}(\sigma))$ obtained from Construction \ref{spin_constr1} as
\begin{equation}
    \label{rochlininvformula}
    R(L_f, \theta_{L_f}(\sigma)) = \Lambda_{L_f} - \theta_{L_f}(\sigma) \cdot \theta_{L_f}(\sigma) + 8\Arf(\theta_{L_f}(\sigma)) \pmod{16}.
\end{equation}
Now, substitute formula (\ref{rochlininvformula}) into Sato's maps, to get:

\begin{lemma}\label{satomapsnew}
    Let $\sigma \in \Spin(\Sigma_g)$ be a spin structure, let $x \in H^1(\Sigma_g)$, and let $[f] \in \Mod_{g,1}[2]$. Let $(L_f, \theta_{L_f}(\sigma))$ and $(L_f, \theta_{L_f}(\sigma + x))$ denote the spin mapping tori obtained from Construction \ref{spin_constr1}, then
\begin{equation*}
    \beta_{\sigma, x}([f])=(\theta_{L_f}(\sigma + x) \cdot \theta_{L_f}(\sigma +x) - \theta_{L_f}(\sigma) \cdot \theta_{L_f}(\sigma) +8(\Arf(\theta_{L_f}(\sigma))-\Arf(\theta_{L_f}(\sigma +x))))/2 \pmod{8}.
\end{equation*}
Here $\theta_{L_f}(\sigma) \cdot \theta_{L_f}(\sigma)$ denotes the sum of all the entries in the linking matrix of the characteristic sublink of $\theta_{L_f}(\sigma)$, and $\Arf(\theta_{L_f}(\sigma)) \in \{ 0, 1 \}$ is the Arf invariant of the characteristic sublink of $\theta_{L_f}(\sigma)$. This formula is well--defined by Theorems \ref{kirby_melvin_rochlin_formula} and \ref{well_defined_sato}, and is independent of the orientations chosen for the components of $L_f$.
\end{lemma}

We choose orientations for $L_f$ to simplify calculations for the terms in the above formula; examples of these orientation choices are given in the proof of Lemma \ref{satoishom} below. Using our description of $\theta_{L_f}$, a spin structure for these surgery diagrams is given by declaring which curves in the framed link $L$ representing $S^1 \times \Sigma_g$ are in the characteristic sublink; the defining relations $C \cdot L_i \equiv L_i \cdot L_i \pmod{2}$ of definition \ref{ffu} determines whether the additional components of $L_f$ are in the characteristic sublink, by Proposition \ref{mon_comp_determined}. For Figure \ref{method1}, giving this diagram a spin structure is equivalent to declaring whether the zero framed components are in the characteristic sublink. The dotted circle component is never in the characteristic sublink for our map $\theta_{L_f}$.

\begin{figure}[h]
\centering
\includegraphics[scale=0.5]{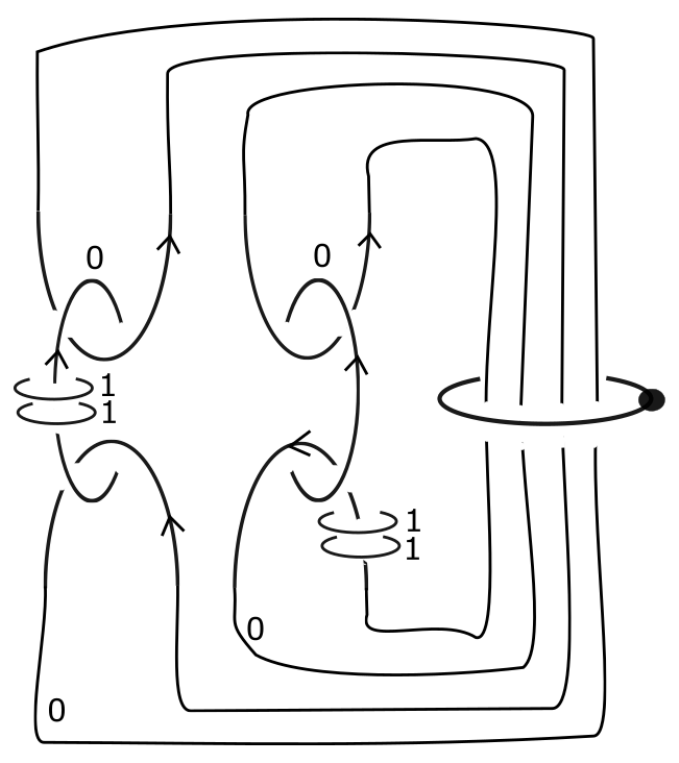}
\caption{Surgery diagram for $M_{t_{a_2}^2t_{b_1}^2}$ obtained from Construction \ref{constr_1}, where $a_2$ and $b_1$ are the curves of Figure \ref{tangle1}}
\label{method1}
\end{figure}

\begin{lemma}
\label{satoishom}
Let $g \geq 3$, let $\sigma \in \Spin(\Sigma_g)$ be a spin structure, and let $\beta_{\sigma}:\Mod_{g,1}[2] \rightarrow \Map(H_1(\Sigma_g),\mathbb{Z}/8)$ denote Sato's maps, where $\beta_{\sigma}([f])(x) = \beta_{\sigma, PD(x)}([f])$, for $PD(x) \in H^1(\Sigma_g)$, and $\beta_{\sigma, PD(x)}([f])$ is given by Lemma \ref{satomapsnew}. Then the maps $\beta_{\sigma}$ are homomorphisms.
\end{lemma}
\begin{proof}
The set $S = \{ t_c^{\pm 2} \; | \; c \; \mathrm{nonseparating} \; \mathrm{simple} \; \mathrm{closed} \; \mathrm{curve} \}$ is a generating set for $\Mod_{g,1}[2]$. Let $x \in H^1(\Sigma_g)$,
we will show that $\beta_{\sigma, x}(f_1\cdots f_n) = \sum_{i=1}^n \beta_{\sigma,x}(f_i)$, for $f_i \in S$. We have a framed link $L$ for $S^1 \times \Sigma_g$ described in Construction \ref{spin_constr1}. We then place curves corresponding to $f_1,..,f_n$ in pushoffs of a fiber surface $F$ that lies in $S^3 \setminus L$.  Note that the framed link $L$ has linking matrix the zero matrix, hence components of $L$ do not contribute to the $\theta_{L_f}(\sigma) \cdot \theta_{L_f}(\sigma)$ terms in Lemma \ref{satomapsnew}.

The components corresponding to the monodromy come in pairs: let $L_c,L_c'$ be two components corresponding to a $t_c^{\epsilon_c}$ factor of $f$, where $\epsilon_c = \pm 2$. By Construction \ref{spin_constr1}, $L_c$ and $L_c'$ can be isotoped to lie in a single fiber surface such that they bound an annulus in this fiber. Now the relations
\begin{equation*}
    p_iw_i + \sum_{j \neq i} lk(L_i, L_j)w_j \equiv p_i \pmod{2}
\end{equation*}
of definition \ref{ffu} applied to $L_c,L_c' \subset L_f$ imply that $L_c$ and $L_c'$ are either both in, or both out, of the characteristic sublink for $\theta_{L_f}(\sigma)$ by Proposition \ref{mon_comp_determined}.

Orient $L_c$ and $L_c'$ oppositely in the fiber, then $L_c$ and $L_c'$ contribute $\epsilon_c$ to the $\theta_{L_f}(\sigma) \cdot \theta_{L_f}(\sigma)$ terms if they are both in the characteristic sublink, and $0$ otherwise; let $\lambda:H_1(F;\Z) \times H_1(F;\Z) \rightarrow \Z$ denote the Seifert linking form for the punctured fiber $F \subset S^3$. In the case $\beta_{\sigma,x}(t_c^{\epsilon_c})$, the linking matrix of the framed link obtained from Construction \ref{spin_constr1}, with the orientation convention given above, is
\begin{equation*}
\begin{bmatrix}
    0 & \dots & \dots & 0 & 0  & 0 \\
    0 & \dots & \dots & 0 & l_1  & -l_1 \\
    0 & \dots & \dots & 0 & l_2 & -l_2 \\
    0 & \ddots & \dots  & 0 & \vdots & \vdots \\
    0 & \dots & \dots & 0 & l_{2g} & -l_{2g} \\
    0 & l_1 & \dots & l_{2g} & m+\epsilon_c/2 & -m \\
    0 & -l_1 & \dots & -l_{2g} & -m & m+\epsilon_c/2 \\
\end{bmatrix},
\end{equation*}
where $m = \lambda(c,c)$, and the $l_i$ are given by the linking numbers between $L_c$ and the components of $L$. For the general case, pairs of rows and columns are adjoined to this linking matrix, corresponding to modifying the monodromy by squares of Dehn twists. Suppose that $t_c^{\epsilon_c}$ and $t_{d}^{\epsilon_d}$ are factors of the monodromy, then the relevant block of the linking matrix corresponding to the components $L_c,L_c',L_d,L_d'$ in the notation/orientation conventions above, has the form
\begin{equation*}
\begin{bmatrix}
    \lambda(c,c)+\epsilon_c/2 & -\lambda(c,c) & l & -l \\
    -\lambda(c,c) & \lambda(c,c) + \epsilon_c/2 & -l & l \\
    l & -l & \lambda(d,d)+ \epsilon_d/2 & -\lambda(d,d) \\
    -l & l & -\lambda(d,d) & \lambda(d,d)+ \epsilon_d/2
\end{bmatrix},
\end{equation*}
where $l = \lambda(c,d)$ or $\lambda(d,c)$, after fixing orientations. So linking with other components corresponding to the monodromy has no effect on the $\theta_{L_f}(\sigma) \cdot \theta_{L_f}(\sigma)$ terms. The proof of Proposition \ref{mon_comp_determined} also implies that whether $L_c$ and $L_c'$ are in the characteristic sublink only depends on the components of $L$, and not on other components corresponding to the monodromy. So the $\theta_{L_f}(\sigma) \cdot \theta_{L_f}(\sigma)$ terms are additive with respect to products of squares of Dehn twists.

For the Arf invariant terms in Lemma \ref{satomapsnew}, note that the Arf invariant is preserved under orientation--preserving band sums. In the convention above, $L_c$ and $L_c'$ bound an annulus, and are oriented oppositely, so after band summing $L_c$ with $L_c'$, we get an unknot that can be isotoped to be disjoint from the rest of the link. Apply this to all pairs of sublinks corresponding to the monodromy, and note that the sublink of $L$ that is in the characteristic sublink is a disjoint union of unknots, since the dotted circle component is never in the characteristic sublink. Therefore, the Arf invariant terms are always zero in Lemma \ref{satomapsnew}.
\end{proof}

A formula for $\beta_{\sigma}$ on squares of Dehn twists on nonseparating curves follows from the proof of Lemma \ref{satoishom}. This formula is analogous to Sato's \autocite[Proposition 5.2]{sato}.
\begin{customcor}{1}
\label{maincor1}
Let $c$ be a nonseparating simple closed curve in $\Sigma_{g,1}$ then we have that 
\begin{equation*}
    \beta_{\sigma, x}(t_c^2) = 
   \left\{
\begin{array}{ll}
      0 , & c_1 \cup c_2 \in \theta_{L_{t_c^2}}(\sigma) \: \mathrm{and} \: \in \theta_{L_{t_c^2}}(\sigma + x) \\
      0 , & c_1 \cup c_2 \notin \theta_{L_{t_c^2}}(\sigma) \: \mathrm{and} \: \notin \theta_{L_{t_c^2}}(\sigma + x) \\
      1 , & c_1 \cup c_2 \in \theta_{L_{t_c^2}}(\sigma + x) \: \mathrm{and} \: \notin \theta_{L_{t_c^2}}(\sigma) \\
      -1 , & c_1 \cup c_2 \in \theta_{L_{t_c^2}}(\sigma) \: \mathrm{and} \: \notin \theta_{L_{t_c^2}}(\sigma + x) ,
\end{array} 
\right.
\end{equation*}
where $c_1$ and $c_2$ are pushoffs of $c$ in the fiber surface, framed using the Seifert pairing via Construction \ref{spin_constr1}, and $c_1 \cup c_2 \in \theta_{L_{t_c^2}}(\sigma)$ denotes that the components corresponding to the monodromy are in the characteristic sublink for Construction \ref{spin_constr1} applied to $\sigma \in \Spin(\Sigma_g)$.
\end{customcor}

\section{The Birman--Craggs maps} \label{birman_craggs_section}
In this section, we use our framework to calculate the Birman--Craggs maps, and to give a proof that these maps are homomorphisms. Then, we calculate Sato's maps on elements of the Torelli group, and find relations between Sato's maps and the Birman--Craggs maps.

\subsection{The Birman--Craggs homomorphisms.}

Birman and Craggs \autocite{birmancraggs} associated to every element $k \in \mathcal{I}_{g,1}$ a $3$-manifold $M(k)$ defined via a Heegaard splitting. They proved that taking the Rochlin invariant produces a family of homomorphisms from the Torelli group to $\mathbb{Z}/2$. Johnson then reformulated the family of homomorphisms in the following way \autocite[Section 5 and 6]{BCJpaper}.

Let $h: \Sigma_g \rightarrow S^3$ be a Heegaard embedding. Split $S^3$ along $h(\Sigma_g)$ into two handlebodies $A$ and $B$. Take $k \in \mathcal{I}_{g,1}$ and reglue $A$ to $B$ along their boundaries by the map $k$ to get the closed $3$-manifold $M(h,k)$. Since $k$ acts trivially on the homology of the Heegaard surface, we have that $M(h,k)$ is a homology $3$-sphere, we then take the Rochlin invariant of the unique spin structure
\begin{equation*}
    \mu(h,k) = R(M(h,k))/8 \pmod{2}.
\end{equation*}
This rewrites every Birman--Craggs homomorphism in the form $\mu(h,-): \mathcal{I} \rightarrow \mathbb{Z}/2$. Johnson was able to enumerate all the maps $\mu(h,-)$ using the Seifert pairing induced by the Heegaard embedding $h$. He collected all these maps into one homomorphism called the Birman--Craggs--Johnson map \autocite[Section 9]{BCJpaper}. This map determines the torsion part of the abelianization of the Torelli group. 

We begin with a model for computing the Birman--Craggs maps using the formula in Theorem \ref{kirby_melvin_rochlin_formula}
\begin{equation}
    \label{combinatorialrochlininv}
    R(L,C) = \Lambda_L - C\cdot C+8\Arf(C) \pmod{16}.
\end{equation}
Here $L$ is a framed link for $M(h,k)$, $C$ is the unique characteristic sublink, $\Lambda_L$ denotes the signature of the linking matrix of $L$, $C \cdot C$ denotes the sum of all entries in the linking matrix of $C$, and $\Arf(C)$ denotes the Arf invariant of $C$.
\subsubsection*{A model for calculating the Birman--Craggs maps.}
First, we describe a Heegaard splitting of $S^3$. Take handlebodies $\pm H_g = \nu(\pm \Delta_g)$, and fix $\Sigma_g = \partial ( \nu (\Delta_g))$ as in Section \ref{heegaard_splitting_subsection}. Let $\{a_i, b_i \}$ denote the standard symplectic basis for $H_1(\Sigma_g ; \mathbb{Z})$, as pictured in Figure \ref{spbasis}. We fix our Heegaard splitting for $S^3$ to be 
\begin{equation*}
    S^3 = H_g \bigcup_{i_g} -H_g = S(i_g),
\end{equation*}
where $i_g \coloneqq \prod_{j=1}^{g}t_{b_j}t_{a_j}t_{b_j}$. Then, 
for any $k \in \mathcal{I}_{g,1}$, we set
\begin{equation*}
    M(h,k) = H_g \bigcup_{i_g \circ k} -H_g = S(i_g \circ k),
\end{equation*}
where $h:\Sigma_g \rightarrow S^3$ denotes the inclusion map. 

We fix our Heegaard embedding $h$ to be the inclusion from now on, and we denote the $3-$manifold $M(h,k)$ described above by $V(k)$.

\begin{figure}
    \centering
    \includegraphics[scale=0.4]{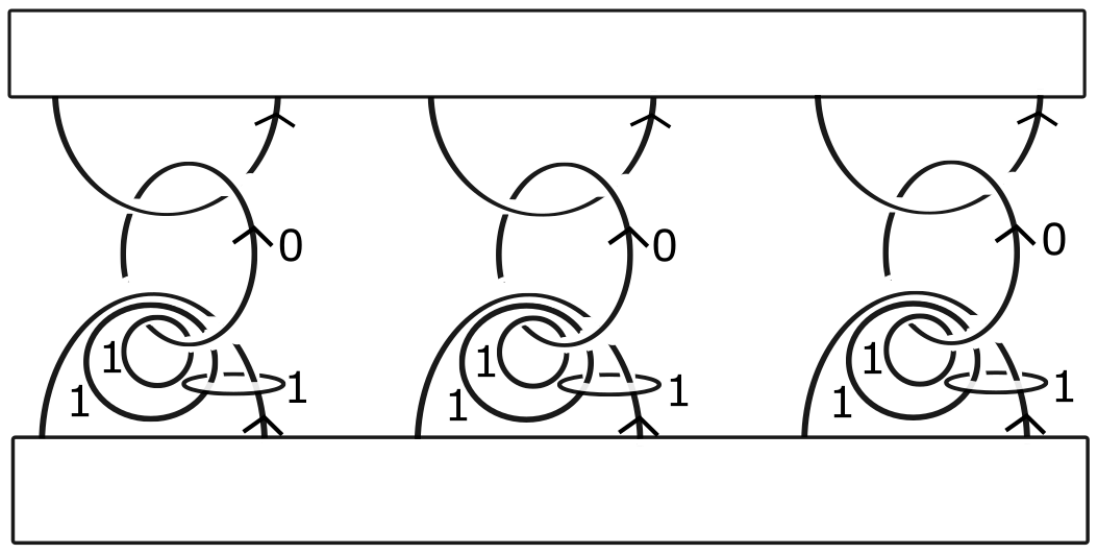}
    \caption{}
    \label{threespheretangle}
\end{figure}
Using Construction \ref{constr_1} (2), we find that our model for $S^3$ is given by Figure \ref{threespheretangle}. 

By Johnson, the Torelli group $\mathcal{I}_{g,1}$ is generated by \textit{bounding pair maps} \autocite[Theorem 1]{johnsonbp_genus_1_generates}. Here a bounding pair $d_1, d_2$ is a pair of simple closed curves on the surface which bound a nontrivial subsurface, and the bounding pair map is given by $t_{d_1}t_{d_2}^{-1}$. Suppose that $k \in \mathcal{I}_{g,1}$ is given as a product of bounding pair maps. We use the method of composing tangles as in Construction \ref{constr_2} (2) to compute $R(V(k))$, as below.

\begin{constr} \label{hom_sphere_constr}
    Suppose $k = f_{n-1} \cdots f_1: \Sigma_g \rightarrow \Sigma_g$ is a diffeomorphism, where each $f_i$ is a bounding pair map, or its inverse.
    \begin{enumerate}
        \item Use Construction \ref{constr_2} (2) with $f = f_nf_{n-1} \cdots f_1$, where $f_n = \iota_g = \prod_{j=1}^g t_{b_j}t_{a_j} t_{b_j}$, and $f_i$ as above for $i \neq n$, to obtain a tangle diagram for $V(k)$. Let $L$ denote the closed components of this diagram; surgery along $L$ gives $V(k)$.
        \item Let $\overline{L}_{g,n}$ denote the closed components obtained from Construction \ref{constr_2} (2) applied to $f_n = \iota_g$ and $f_i = \mathrm{id}$ for $i=1,..,n-1$. Then surgery on $\overline{L}_{g,n}$ gives $S^3$, and $\overline{L}_{g,n}$ is always a sublink of $L$. The relations of definition \ref{ffu} imply that the unique characteristic sublink of $\overline{L}_{g,n}$ is the union of the components corresponding to the $t_{b_j}$ factors of $\iota_g$ in the surgery diagram for $S(\iota_g)$ obtained from Construction \ref{constr_2} (1); see Figure \ref{bcfiguretwo} for the $n=g=3$ case, and compare this with Figures \ref{spbasis} and \ref{threespheretangle}. 
        \item The components of $L - \overline{L}_{g,n}$ come in pairs corresponding to the bounding pairs $f_i$, $i=1,..,n-1$ via Construction \ref{constr_2} (1). Let $L_i,L_{i+1}$ be a pair of components of $L$ corresponding to the bounding pair $t_dt_e^{-1}$ in the factorisation of $k$. Now $[d] = [e] \in H_1(\Sigma_g;\Z)$, and the linking numbers of $L_i, L_{i+1}$ with the other components is determined by the Seifert pairing $\lambda$ for $\Sigma_g = \partial( \nu (\Delta_g)) \subset S^3$. So $L_i$ and $L_{i+1}$ have the same linking number with the other components of $L$ up to sign. Using definition \ref{ffu} we have $p_iw_i + (p_i-1)w_{i+1} + \sum_{j \neq i,i+1} lk(L_i,L_j) w_j \equiv p_i \pmod{2}$ and $p_{i+1} w_{i+1} + (p_{i+1}-1)w_i + \sum_{j \neq i, i+1} lk(L_i, L_j)w_j \equiv p_{i+1} \pmod{2}$. By (2), the components of $\overline{L}_{g,n}$ in the characteristic sublink are disjoint from $L_i,L_{i+1}$, so $\sum_{j \neq i,i+1}lk(L_i,L_j)w_j = 0$. Since $p_i = \lambda([d],[d])+1$ and $p_{i+1} = \lambda([d],[d])-1$, we conclude that $w_i \equiv w_{i+1} \equiv \lambda([d],[d]) + 1 \pmod{2}$. Therefore, $L_i$ and $L_{i+1}$ are in the characteristic sublink if and only if $\lambda([d],[d])$ is even.
    \end{enumerate}
\end{constr}

\begin{figure}
\centering
\includegraphics[scale=0.5]{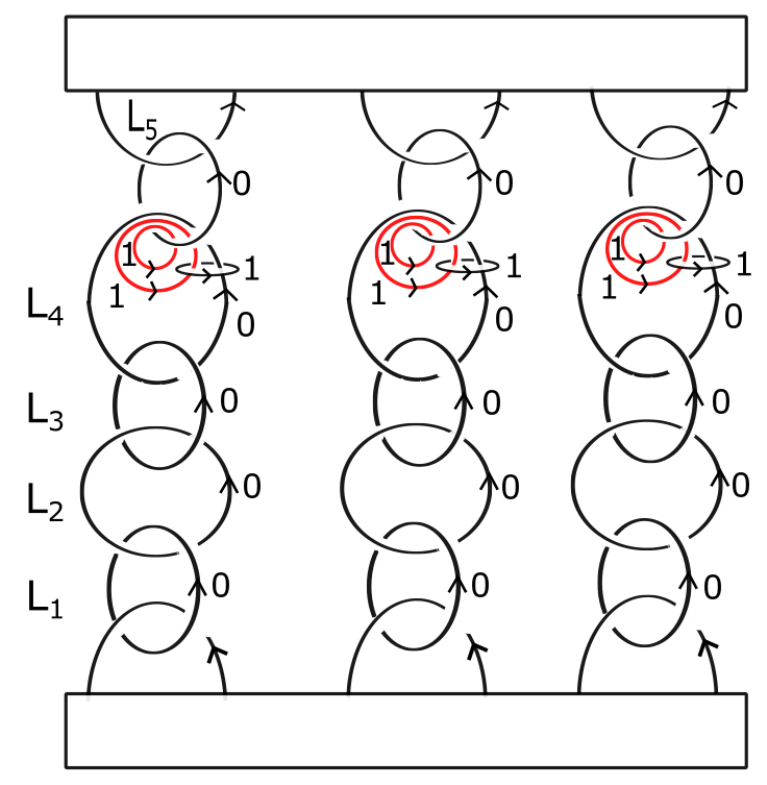}
\caption{The framed link $\overline{L}_{3,3}$, the coloured components give the characteristic sublink.}
\label{bcfiguretwo}
\end{figure}

To show that the Birman--Craggs maps are homomorphisms, we must first analyse the linking matrix of the sublink $\overline{L}_{g,n}$ obtained in Construction \ref{hom_sphere_constr}.

\begin{lemma} \label{L_bar_signature}
    The signature of the linking matrix of $\overline{L}_{g,n}$, using the orientation convention of Figure \ref{bcfiguretwo}, is $2g$.
\end{lemma}
\begin{proof}
    Note that $\overline{L}_{g,n}$ is the disjoint union of $g$ copies of $\overline{L}_{1,n}$; c.f. Figure \ref{bcfiguretwo}. In the orientation convention of Figure \ref{bcfiguretwo}, one can show, using Sylvester's law of inertia, that $$\Lambda_{\overline{L}_{1,n}} = \mathrm{Sign}\begin{pmatrix}
        0 & -1 \\
        -1 & 0
    \end{pmatrix} + \Lambda_{\overline{L}_{1,n-1}}.$$
    For example, add the first row to the third, then add the first column to the third in the linking matrix of $\overline{L}_{1,n}$ in the ordering of Figure \ref{bcfiguretwo}. Since $\Lambda_{\overline{L}_{1,1}} = \Lambda_{\overline{L}_{1,2}} = 2$, and $\mathrm{Sign}\begin{pmatrix}
        0 & -1 \\
        -1 & 0
    \end{pmatrix} = 0$, we conclude that $\Lambda_{\overline{L}_{1,n}} = 2$, therefore $\Lambda_{\overline{L}_{g,n}} = 2g$.
\end{proof}

\begin{theorem}
\label{bcishom}
\autocite[Theorem 8]{birmancraggs}
Let $g \geq 3$. Let $\mu: \mathcal{I}_{g,1} \rightarrow \mathbb{Z}/2$ be given by $k \mapsto R(L,C)$, where $(L,C)$ is obtained from Construction \ref{hom_sphere_constr} applied to $k$, and $R(L,C)$ is the formula given in Theorem \ref{kirby_melvin_rochlin_formula}. Then $\mu$ is a well--defined homomorphism.
\end{theorem}
\begin{proof}
To show that $\mu$ is well--defined, note that for any two factorisations of $k$ into bounding pairs, the manifolds obtained from Construction \ref{hom_sphere_constr} are diffeomorphic by Lemma \ref{composingtangles}. There is only one spin structure, so the corresponding values of $R(L,C)$ must be equal modulo $16$, by Theorem \ref{kirby_melvin_rochlin_formula}. 

To show that $\mu$ is a homomorphism, we analyse the pair $(L,C)$ obtained from Construction \ref{hom_sphere_constr}. Additivity of the $8\Arf(C)$ terms in formula (\ref{combinatorialrochlininv}) follows from the components corresponding to the bounding pairs in Construction \ref{hom_sphere_constr} being disjoint from each other.

To deal with the $\Lambda_L - C \cdot C$ terms in formula (\ref{combinatorialrochlininv}), we orient a pair of components $L_i$ and $L_{i+1}$ corresponding to a bounding pair factor $t_dt_e^{-1}$, as in Construction \ref{hom_sphere_constr} (3), oppositely on the Heegaard surface. Fix the orientation conventions for the components $\overline{L}_{g,n}$ obtained from Construction \ref{hom_sphere_constr} (2) as in Figure \ref{bcfiguretwo}. The linking matrix for the closed components in the tangle diagram for $V(k)$ obtained from Construction \ref{hom_sphere_constr} can be written as 
\begin{equation*}
R=
\begin{pmatrix}
    \overline{A} & \dots & \dots & 0  & A_{1,1}  & A_{2,1} & \vdots & A_{n-1,1} \\
    0 & \ddots & \vdots & 0 & A_{1,2}  & A_{2,2} & \vdots & A_{n-1,2}\\
    0 & \dots & \ddots & \vdots & \vdots & \vdots & \vdots & \vdots \\
    0 & 0 & \dots  & \overline{A} & A_{1,g} & A_{2,g} & \dots & A_{n-1,g}  \\
    A_{1,1}' & A_{1,2}' & \dots &  A_{1,g}' & A_1 & 0 & \dots & 0 \\
    A_{2,1}' & A_{2,2}' & \dots &  A_{2,g}' & 0 & A_2 & \dots & 0 \\
    \vdots & \vdots & \vdots & \vdots & 0 & 0 & \ddots & 0 \\
    A_{n-1,1}' & A_{n-1,2}' & \dots & A_{n-1,g}' & 0 & 0 & \dots & A_{n-1}
\end{pmatrix}.
\end{equation*}
The top diagonal is the linking matrix of $\overline{L}_{g,n}$, which is the direct sum of $g$ copies of the linking matrix $\overline{A}$ of $\overline{L}_{1,n}$. Note that $\overline{L}_{1,n}$ has $2(n-1)+4$ components. The $A_i$ are the linking matrices of a pair of components corresponding to a bounding pair factor $f_i$ in Construction \ref{hom_sphere_constr}. Suppose that $f_i = t_dt_e^{-1}$, then 
\begin{equation*}A_i = 
    \begin{pmatrix}
        \lambda([d],[d]) + 1 & -\lambda([d],[d]) \\
        -\lambda([d],[d]) & \lambda([d],[d]) -1
    \end{pmatrix},
\end{equation*}
in our orientation convention. The $A_{i,j}$ are $(2(n-1)+4) \times 2$ matrices corresponding to how the pair of components coming from $f_i$ link with the components of $\overline{L}_{g,n}$. In our orientation convention, the $A_{i,j}$ have identical columns but with opposite signs. The $A_{i,j}'$ are the transpose of the $A_{i,j}$, so they have identical rows but with opposite signs.

Recall that if $E$ is an elementary matrix for a row operation, then $ERE^T$ is obtained from the matrix $R$ by simultaneous row and column operations. Sylvester's law of inertia states that the signatures of $R$ and $ERE^T$ are equal. We repeatedly apply these operations to get a matrix of the form
\begin{equation*}
R'=
\begin{bmatrix}
    \overline{A} & \dots & \dots & 0 & 0 & 0 & \dots & 0 \\
    0 & \ddots & \dots & 0 & 0 & 0 & \dots & 0 \\
    0 & \dots & \ddots & 0 & 0 & 0 & \dots & 0 \\
    0 & 0 & \dots  & \overline{A} & 0 & 0 & \dots  & 0 \\
    0 & \dots & \dots & 0 & A_1' & 0 & \dots & 0 \\
    0 & 0 & \dots & 0 & 0 & A_2' & \dots & 0 \\
    0 & 0 & \dots & 0 & 0 & 0 & \ddots &  0 \\
    0 & 0 & \dots & 0 & 0 & 0 & 0 & A_n'
\end{bmatrix}.
\end{equation*}
Here the signature of the matrices $A_i'$ are the same as the signature of the matrices $A_i$. Since the signature of any matrix of the form $A_i$ is zero, we get that $\Lambda_L = 2g$ by Lemma \ref{L_bar_signature}. The only components of $C$ that contribute to the $C \cdot C$ terms are the ones contained in $\overline{L}_{g,n}$. So we have $\Lambda_L-C \cdot C = 2g-2g = 0$.
\end{proof}

\subsection{Evaluation on separating twists} \label{sep_eval_subsection}
Now we obtain formulas for evaluating Sato's maps on elements of the Torelli group. This relates Sato's maps to the Birman--Craggs maps.

We need the following evaluation of $\mu(h,-)$, due to Johnson. Here $h$ is the inclusion, so we think of the Heegaard surface as being in $S^3$.
\begin{theorem}
\label{johnsonsformula}
\autocite[Theorem 1]{BCJpaper}
Let $\gamma$ be a separating simple closed curve on the Heegaard surface, then
\begin{equation*}
    \mu(h, t_{\gamma}) = \Arf(\gamma).
\end{equation*}
Here $\Arf(\gamma)$ can be evaluated by taking the subsurface $S \subset \Sigma_g$ which $\gamma$ bounds and calculating the Arf invariant of the quadratic form induced by the Seifert pairing on the Heegaard surface.
\end{theorem}
Now we relate Sato's maps to the Birman--Craggs maps.
\begin{corollary}
\label{sepformula}
Let $\gamma$ be a separating simple closed curve in $\Sigma_g$, then 
\begin{equation*}
\beta_{\sigma, x}(t_{\gamma}) = 4(\Arf(\theta_{L_{t_{\gamma}}}(\sigma))-\Arf(\theta_{L_{t_{\gamma}}}(\sigma +x))).
\end{equation*}
\end{corollary}
\begin{proof}
Apply Construction \ref{constr_1} to $f=t_{\gamma}$ to get a framed link $L_{t_{\gamma}}$ for the mapping torus $M_{t_{\gamma}}$. Following Construction \ref{spin_constr1}, we can define a characteristic sublink $\theta_{L_{t_{\gamma}}}(\sigma)$. The method of proof of Theorem \ref{well_defined_sato} implies that the spin manifold $(L_{t_{\gamma}}, \theta_{L_{t_{\gamma}}}(\sigma))$ obtained this way is spin diffeomorphic to Construction \ref{spin_constr1} applied to any factorisation of $t_{\gamma}$ into squares of Dehn twists. We evaluate $\beta_{\sigma,x}(t_{\gamma})$ using Theorem \ref{kirby_melvin_rochlin_formula} applied to $(L_{t_{\gamma}}, \theta_{L_{t_{\gamma}}}(\sigma))$ and $(L_{t_{\gamma}}, \theta_{L_{t_{\gamma}}}(\sigma+x))$ as in Lemma \ref{satomapsnew}.

Let $L$ be the framed link for $S^1 \times \Sigma_g$ obtained from Construction \ref{constr_1} (3) with $f=\mathrm{id}$, then $L$ is a sublink of $L_{t_{\gamma}}$. Let $F$ be the fiber surface pictured in Figure \ref{vis_fibers}. The components of $L$ can all be isotoped in $S^3$ to the canonical basis for $H_1(F ; \mathbb{Z})$ in the fiber.
Let $L_i$ be the component of $L_{t_{\gamma}} - L$ that corresponds to the monodromy. Then the linking number of $L_i$ with any of the components of $L$ is determined by the Seifert pairing $\lambda$ of the surface $F \subset S^3$. Since $[\gamma]=0 \in H_1(F;\Z)$, $L_i$ has linking number $0$ with any component of $L$. This implies that the linking matrix of $L_{t_{\gamma}}$ is zero everywhere except for one entry on the main diagonal, which is the framing of the curve $L_i$. We compute that this framing is given by $\lambda(\gamma, \gamma) +1 = \lambda(0,0)+1=1$. For any spin structure $\sigma \in \Spin(\Sigma)$, using the relations of definition \ref{ffu} we have that 
\begin{equation*}
    w_i + 0 \equiv 1 \pmod{2}.
\end{equation*}
So $L_i$ is always in the characteristic sublink. This gives us that $\theta_{L_{t_{\gamma}}}(\sigma + x) \cdot \theta_{L_{t_{\gamma}}}(\sigma +x) - \theta_{L_{t_{\gamma}}}(\sigma) \cdot \theta_{L_{t_{\gamma}}}(\sigma) = 0$ always, so we have that 
\begin{equation*}
    \beta_{\sigma, x}(t_{\gamma}) = 4(\Arf(\theta_{L_{t_{\gamma}}}(\sigma))-\Arf(\theta_{L_{t_{\gamma}}}(\sigma +x))).
\end{equation*}
\end{proof}

The proof of Corollary \ref{sepformula} gives $\Lambda_{L_{t_{\gamma}}} - C \cdot C = 1-1 =0$, hence we obtain
\begin{equation*}
    R(M_{L_{t_{\gamma}}}, \theta_{L_{t_{\gamma}}}(\sigma)) = 8\Arf(\theta_{L_{t_{\gamma}}}(\sigma)) \pmod{16},
\end{equation*}
when $\gamma$ is separating. Let $\eta$ denote the spin structure on $\Sigma_g$ with $\theta_{L_{t_{\gamma}}}(\eta)$ containing none of the components of the link $L$ that represents $S^1 \times \Sigma_g$. We get 
\begin{equation*}
    R(M_{L_{t_{\gamma}}}, \theta_{L_{t_{\gamma}}}(\eta)) = 8\Arf(\gamma) \pmod{16},
\end{equation*}
the right side can be evaluated using the cut surface of $\gamma \subset F= \Sigma_{g,1}$ and the quadratic form on $H_1(F)$ induced by the Seifert pairing $\lambda$ of $F \hookrightarrow S^3$.

Note that for $[h] \in \Mod_{g,1}[2]$ and $[f] \in \Mod_{g,1}$, the diffeomorphism $\mathrm{id} \times f$ descends to a diffeomorphism $M_h \rightarrow M_{fhf^{-1}}$ of mapping tori. Under this diffeomorphism, the spin structure $\theta(\sigma)$ on $M_{fhf^{-1}}$ pulls back to the spin structure $\theta(f^*\sigma)$ on $M_h$. This implies that 
\begin{equation*}
    R(M_{fhf^{-1}}, \theta(\sigma)) = R(M_h, \theta(f^*\sigma)),
\end{equation*}
for all $[f] \in \Mod_{g,1}$ and $[h] \in \Mod_{g,1}[2]$ (in the notation of Section \ref{sato_overview_section}). 

Johnson proved that $\Spin(\Sigma_g)$ is in bijection with symplectic quadratic forms (see Theorem \ref{spin_strucures_quadratic_forms_bij}). Arf showed that two symplectic quadratic forms are in the same orbit under the symplectic group if and only if they have the same Arf invariant. So there are exactly two $\Mod_{g,1}$ orbits in $\Spin(\Sigma_g)$. Suppose we chose $x \in H^1(\Sigma_g)$ such that $\eta +x \in \Spin(\Sigma_g)$ is in the same orbit as $\eta$; write $\eta + x = f^*\eta$ for some $[f] \in \Mod_{g,1}$, then
\begin{equation*}
    \beta_{\sigma, x}(t_{\gamma}) = (R(M_{L_{t_{\gamma}}}, \theta_{L_{t_{\gamma}}}(\eta)) - R(M_{L_{t_{\gamma}}}, \theta_{L_{t_{\gamma}}}(f^*\eta)))/2 = (R(M_{L_{t_{\gamma}}}, \theta_{L_{t_{\gamma}}}(\eta)) - R(M_{L_{ft_{\gamma}f^{-1}}}, \theta_{L_{ft_{\gamma}f^{-1}}}(\eta)))/2,
\end{equation*}
where $ft_{\gamma}f^{-1} = t_{f(\gamma)}$, and $f(\gamma)$ is also a separating curve of the same genus on the fiber surface. So we have
\begin{equation*}
    \beta_{\sigma, x}(t_{\gamma}) = 4(\Arf(\gamma) - \Arf(f(\gamma))) \pmod{8}.
\end{equation*}
Note that surgery along $\gamma \subset \Sigma_{g,1}$ with framing $1$ is equivalent to cutting $S^3$ open along a Heegaard surface and regluing by $t_{\gamma}$. Since $t_{\gamma}$ acts trivially on homology, this gives a homology sphere. So this relates the Rochlin invariant of mapping tori to the Rochlin invariant of homology spheres, and to Johnson's description of the Birman-Craggs homomorphisms \autocite{BCJpaper}. In summary, if we combine Theorem \ref{johnsonsformula} and Corollary \ref{sepformula} with the discussion above, we get the following.
\begin{customcor}{2}
\label{maincor2}
Let $c$ be a separating curve on $\Sigma_{g,1}$. Let $\eta$ be the spin structure on $\Sigma_g$ with the characteristic sublink of $\theta_{L_{t_c}}(\eta)$ containing none of the components from the link $L$ that represents $S^1 \times \Sigma_g$. Suppose $\sigma = f^*(\eta)$ and $\sigma + x = h^*(\eta)$ for $[f],[h] \in \Mod_{g,1}$ , then
    \begin{equation*}
        \beta_{\sigma, x}(t_c) = \mu(t_{f(c)}) - \mu(t_{h(c)}) \pmod{2},
    \end{equation*} where $\mu$ denotes the Birman-Craggs homomorphism for the standard embedding $\Sigma_{g} \hookrightarrow S^3$.
\end{customcor}

\subsection{Evaluation on bounding pairs} \label{bp_eval_subsection}
A bounding pair is a pair of disjoint, homologous, non-separating simple closed curves $a,b$ on $\Sigma_{g,1}$, and the bounding pair map is $f = t_at_b^{-1}$. To calculate $\beta_{\sigma,x}(t_at_b^{-1})$, use Construction \ref{constr_1} with $f=t_at_b^{-1}$ to obtain a framed link $L_f$ for $M_f$. Apply the method of Construction \ref{spin_constr1} to obtain the pair $(L_f, \theta_{L_f}(\sigma))$ for any $\sigma \in \Spin(\Sigma_g)$. The manifold $(L_f, \theta_{L_f}(\sigma))$ obtained this way is spin diffeomorphic to any manifold obtained from Construction \ref{spin_constr1} applied to any factorisation of $f$ into squares of Dehn twists; c.f. Lemma \ref{satomapsnew}. Hence we can calculate $\beta_{\sigma, x}(f)$ by using Theorem \ref{kirby_melvin_rochlin_formula} to get
\begin{equation*}
    \beta_{\sigma, x}(f)=(\theta_{L_f}(\sigma + x) \cdot \theta_{L_f}(\sigma +x) - \theta_{L_f}(\sigma) \cdot \theta_{L_f}(\sigma) +8(\Arf(\theta_{L_f}(\sigma))-\Arf(\theta_{L_f}(\sigma +x))))/2 \pmod{8}.
\end{equation*}

Let $L$ be the sublink of $L_f$ corresponding to $S^1 \times \Sigma_g$, and let $F$ be a punctured fiber surface pictured as in Figure \ref{vis_fibers}. Suppose that $\lambda([a],[a]) = \lambda([b],[b]) = \lambda([a],[b]) = lk(a,b) = m$, where $\lambda(-,-)$ is the Seifert linking form for $F \subset S^3$. The framed link $L_f$ is obtained from $L$ by placing the curves $a,b$ in $F$, and framing them using the Seifert pairing, to get $L_i,L_{i+1}$. For any component $L_j$ of $L$, we have $lk(L_j,L_i) = lk(L_j,L_{i+1})$. To specify $(L_f, \theta_{L_f}(\sigma))$, we fix which components of $L$ are in the characteristic sublink, and use definition \ref{ffu} to find the full characteristic sublink. We have
\begin{equation*}
    (m+1) w_i + \sum_{j \neq i,i+1}lk(L_i,L_j)w_j + mw_{i+1} \equiv (m-1)w_{i+1} + \sum_{j \neq i,i+1}lk(L_i,L_j)w_j + mw_i \equiv m+1 \pmod{2},
\end{equation*}
which simplifies to 
\begin{equation}
    \label{boundingpairrelations}
    (m+1)w_i + mw_{i+1} = (m-1)w_{i+1} + mw_i \equiv m+1 - \sum_{j \neq i,i+1}lk(L_i,L_j)w_j \pmod{2}.
\end{equation}

We see that $L_i$ and $L_{i+1}$ are always either both in, or both out of any characteristic sublink of $L_f$. Now orient $L_i$ and $L_{i+1}$ oppositely in the fiber surface, this does not change the framings of $L_i, L_{i+1}$, but changes the sign of the linking number of $L_{i+1}$ with every other component. Then, up to the ordering of the components, the linking matrix for $L_f$ is
\begin{equation*}
\begin{bmatrix}
    0 & \dots & \dots & 0 & 0  & 0 \\
    0 & \dots & \dots & 0 & l_1  & -l_1 \\
    0 & \dots & \dots & 0 & l_2 & -l_2 \\
    0 & \ddots & \dots  & 0 & \vdots & \vdots \\
    0 & \dots & \dots & 0 & l_{2g} & -l_{2g} \\
    0 & l_1 & \dots & l_{2g} & m+1 & -m \\
    0 & -l_1 & \dots & -l_{2g} & -m & m-1 \\
\end{bmatrix}.
\end{equation*}

Suppose that $w_i \equiv w_{i+1} \equiv 1$, then since $\theta_{L_{f}}(\sigma) \cdot \theta_{L_{f}}(\sigma)$ is the sum of the entries in the linking matrix of the characteristic sublink specified by $\theta_{L_{f}}(\sigma)$, we have that 
\begin{equation*}
    \theta_{L_f}(\sigma) \cdot \theta_{L_f}(\sigma) = m+1 + m-1 - 2m = 0.
\end{equation*}
If $w_i \equiv w_{i+1} \equiv 0$, then $\theta_{L_f}(\sigma) \cdot \theta_{L_f}(\sigma) = 0$. These are all the cases, so we have 
\begin{equation*}
    \beta_{\sigma, x}(t_a t_b^{-1}) = 4(\Arf(\theta_{L_{t_at_b^{-1}}}(\sigma)) - \Arf(\theta_{L_{t_at_b^{-1}}}(\sigma +x))) \pmod{8}.
\end{equation*}

Let $\eta \in \Spin(\Sigma_g)$ denote the spin structure with the characteristic sublink of $\theta_{L_{t_at_b^{-1}}}(\eta)$ containing none of the components from the link $L$ for $S^1 \times \Sigma_g$. Using the relations of (\ref{boundingpairrelations}), we have that if $m = lk(a,b)$ is even then 
\begin{equation*}
    w_i \equiv w_{i+1} \equiv 1 \pmod{2},
\end{equation*}
so $L_i$ and $L_{i+1}$ are in the characteristic sublink, and we have 
\begin{equation*}
    R(M_{L_{t_at_b^{-1}}}, \theta_{L_{t_at_b^{-1}}}(\eta)) = \Lambda_{L_{t_at_b^{-1}}} + 8\Arf(a \cup b) = 8\Arf(a \cup b),
\end{equation*}
where $\Lambda_{L_{t_at_b^{-1}}}$ is the signature of the linking matrix given above. Here $\Arf(a \cup b)$ can be evaluated using the cut surface of $a \cup b \subset F = \Sigma_{g,1}$ and the quadratic form induced by the Seifert pairing $\lambda$.

If $m$ is odd then
\begin{equation*}
    mw_{i+1} = 0 \pmod{2},
\end{equation*}
so $L_i$ and $L_{i+1}$ are not in the characteristic sublink, and we have 
\begin{equation*}
    R(M_{L_{t_at_b^{-1}}}, \theta_{L_{t_at_b^{-1}}}(\eta)) = \Lambda_{L_{t_at_b^{-1}}} = 0.
\end{equation*}
For $a \cup b \subset \Sigma_{g,1}$, surgery with coefficients $m \pm 1$ is equivalent to cutting $S^3$ open along a Heegaard surface, and regluing by $t_at_b^{-1}$. Since $t_at_b^{-1}$ acts trivially on homology the resulting space is a homology sphere. So we obtain a relation between Rochlin invariants of a mapping tori and Rochlin invariants of homology spheres.

Now we relate Sato's maps on bounding pairs to the Birman--Craggs maps using the same method as in the case of separating curves. Suppose we have a spin structure $\sigma \in \Spin(\Sigma_g)$ and a class $x \in H^1(\Sigma_g)$ such that there exists mapping classes $[f],[h] \in \Mod_{g,1}$ with $\sigma = f^*(\eta)$ and $\sigma + x = h^*(\eta)$. We have, in the notation of Section \ref{sato_overview_section},
\begin{align*}
    \beta_{\sigma, x}(t_at_b^{-1}) = (R(M_{t_at_b^{-1}}, \theta(f^*(\eta))) - R(M_{t_at_b^{-1}}, \theta(h^*(\eta))))/2
    = (R(M_{t_{f(a)}t_{f(b)}^{-1}}, \theta(\eta))) - R(M_{t_{h(a)}t_{h(b)}^{-1}}, \theta(\eta)))/2.
\end{align*}
We also need the following calculation of the Birman--Craggs maps, due to Johnson; this calculation also follows from the proof of Theorem \ref{bcishom}.
\begin{lemma}\autocite[Theorem 1]{BCJpaper}
\label{bcbpformula}
Let $a,b$ be a pair of simple closed curves on $\Sigma_{g,1}$ that bound a subsurface, and let $\lambda: H_1(\Sigma_g ; \mathbb{Z})\times H_1(\Sigma_g ; \mathbb{Z}) \rightarrow \mathbb{Z}$ denote the Seifert linking form for the standard embedding $\Sigma_{g} \hookrightarrow S^3$. Then for the Birman-Craggs map $\mu: \mathcal{I}_{g,1} \rightarrow \mathbb{Z}/2$ corresponding to the embedding $\Sigma_{g} \hookrightarrow S^3$ we have that:
\begin{enumerate}
    \item $\mu(t_at_b^{-1}) = 0$, if $\lambda(a,a)$ is odd,
    \item $\mu(t_at_b^{-1}) = 8\Arf(a \cup b)$, if $\lambda(a,a)$ is even.
\end{enumerate}
\end{lemma}
In summary, if we combine Lemma \ref{bcbpformula} with the calculations of this subsection, we get the following.
\begin{customcor}{3}
\label{maincor3}
Let $a,b$ be a pair of simple closed curves on $\Sigma_{g,1}$ that bound a subsurface. Let $\eta$ be the spin structure on $\Sigma_g$ with the characteristic sublink of $\theta_{L_{t_at_b^{-1}}}(\eta)$ containing none of the components from the link $L$ that represents $S^1 \times \Sigma_g$. If $\sigma = f^*(\eta)$ and $\sigma + x = h^*(\eta)$ for $[f],[h] \in \Mod_{g,1}$, then
\begin{equation*}
    \beta_{\sigma, x}(t_at_b^{-1}) = \mu_{\iota}(t_{f(a)}t_{f(b)}^{-1})-\mu_{\iota}(t_{h(a)}t_{h(b)}^{-1}) \pmod{2},
\end{equation*}
where $\mu_{\iota}$ denotes the Birman--Craggs map for the standard embedding $\iota:\Sigma_g \hookrightarrow S^3$. In particular, if $g \geq 3$, then we have $\beta_{\sigma,x} = \mu_{\iota \circ f}-\mu_{\iota \circ h} \pmod{2}$.
\end{customcor}
\section{Relation to Meyer's signature cocycle} \label{meyer_cocycle_section}

In this section, we apply the methods developed above to study Meyer's signature cocycle restricted to $\Mod_{g,1}[2]$. The extension of the Birman--Craggs maps to $\Mod_{g,1}[2]$ no longer give homomorphisms, but we can find a relation between these extensions and Meyer's signature cocycle.

Let $P$ denote a pair of pants, that is, a sphere with three boundary components. Pick two based loops that run around one distinct boundary component each, call them $\alpha, \beta$, and identify $\pi_1(P)$ with the free group generated by $\alpha, \beta$. Let $[f], [h] \in \Mod_{g,1}$ be mapping classes, and consider the $\Sigma_g$--bundle $E_{f, h}$ over $P$ with monodromy $\rho: \pi_1(P) \rightarrow \Mod_{g,1}$ given by $\alpha \mapsto [f], \beta \mapsto [h]$. The diffeomorphism type of $E_{f, h}$ does not depend on the choice of representatives for the mapping classes $[f], [h]$. Furthermore, $E_{f,h}$ has a natural orientation coming from that of $\Sigma_g$ and $P$.

\textit{Meyer's signature cocycle} is defined by 
\begin{align*}
\tau_g: \Mod_{g,1} \times \Mod_{g,1} \rightarrow \Z \\
([f], [h]) \mapsto \mathrm{Sign}(E_{f, h}),
\end{align*}
where $\mathrm{Sign}(E_{f, h})$ denotes the signature of the $4$--manifold $E_{f,h}$ \cite{meyer}. Note that $\partial E_{f, h} = M_{f} \sqcup M_{h} \sqcup M_{(f \circ h)^{-1}}$. Now, Construction \ref{spin_constr1} and Theorem \ref{well_defined_sato} give us a well--defined map 
\begin{align*}
    R_{\sigma}: \Mod_{g,1}[2] \rightarrow \Z/16 \\
    [f] \mapsto R(M_{L_f}, \theta_{L_f}(\sigma)),
\end{align*}
where $R(M_{L_f}, \theta_{L_f}(\sigma))$ denotes the Rochlin invariant of the spin $3$--manifold $(M_{L_f}, \theta_{L_f}(\sigma))$ obtained from Construction \ref{spin_constr1}. Sato shows in \cite[Lemma 2.2]{sato} that for any $\sigma \in \Spin(\Sigma_g)$, there exists a spin structure on $E_{f, h}$ that spin bounds the mapping tori $M_{f} \sqcup M_{h} \sqcup M_{(f \circ h)^{-1}}$ with spin structure $\theta(\sigma)$ (in the notation of Section \ref{sato_overview_section}). Then we have
\begin{align*}
    R(M_{f}, \theta(\sigma))+R(M_{h}, \theta(\sigma)) - R(M_{f \circ h}, \theta(\sigma)) \equiv \mathrm{Sign}(E_{f, h}) \pmod{16},
\end{align*}
or, in the language of group cohomology $\tau_g \equiv \partial(R_{\sigma}) \pmod{16}$, for any $\sigma \in \Spin(\Sigma_g)$. 

In Construction \ref{spin_constr1}, we start with a framed link $L$ representing $S^1 \times \Sigma_g$. We have the monodromy $f$ written as a product of Dehn twists, and our method gives a framed link $L_f$ representing $M_f$, obtained by placing the curves appearing in the factorisation in pushoffs of the fiber surface in $S^3 \setminus L$, framed using the Seifert pairing; $L_f$ has $L$ as a sublink. Let $\eta$ denote the spin structure on $\Sigma_g$ such that the characteristic sublink of $\theta_{L_f}(\eta)$ contains none of the components from the framed link $L$. For any component added to our diagram of $S^1 \times \Sigma_g$ to modify the monodromy, the relations of definition \ref{ffu} determine whether this component is added to the characteristic sublink.

We relate Construction \ref{spin_constr1} to the Birman--Craggs map for the standard inclusion $\iota: \Sigma_g \hookrightarrow S^3$. Recall Johnson's definition of this map
\begin{align*}
    \mu_{\iota}: \Mod_{g,1}[2] \rightarrow \Z/16 \\
    [f] \mapsto R(M(\iota , f)),
\end{align*}
where $R(M(\iota, f))$ is the Rochlin invariant of the manifold $M(\iota, f)$, obtained by cutting $S^3$ along $\iota(\Sigma_g)$ into two handlebodies, and regluing them along their boundaries by the map $f$. Since $[f] \in \Mod_{g,1}[2]$, the manifold $M(\iota, f)$ is a $\Z/2$--homology sphere, so it has a unique spin structure.

Let $\widetilde{L}_f = L_f \setminus L$ denote the framed link given by the union of the components corresponding to the monodromy $f$. The fiber surface in our diagram for $S^1 \times \Sigma_g$ lies in $S^3$ as the standard embedding $\iota: \Sigma_{g,1} \hookrightarrow S^3$ by Figure \ref{vis_fibers}. Surgery along the framed link $\widetilde{L}_f$ is equivalent to cutting $S^3$ along this Heegaard surface, and regluing via $[f] \in \Mod_{g,1}[2]$, which gives the $\Z/2$--homology sphere $M(\iota, f)$. We begin relating $R_{\eta}$ and $\mu_{\iota}$ with the following Lemma.

\begin{lemma} \label{bc_extension_char_sublink}
The characteristic sublink of $\widetilde{L}_f$ for $M(\iota, f)$ coincides with the characteristic sublink for $(L_f, \theta_{L_f}(\eta))$ obtained from Construction \ref{spin_constr1}, where $\eta$ is the spin structure of $\Sigma_g$ defined above. For a pair of components $L_i$ and $L_{i+1}$ in $\widetilde{L}_f$ that corresponds to a factor of the form $t_a^{\pm 2}$ appearing in $f$, we have that $L_i$ and $L_{i+1}$ are in the characteristic sublink of $\widetilde{L}_f$ if and only if $\lambda([a],[a])$ is even. Here, $\lambda$ denotes the Seifert pairing for $\iota: \Sigma_g \hookrightarrow S^3$.
\end{lemma}
\begin{proof} Take a pair of components $L_i, L_{i+1}$ in $\widetilde{L}_f$, that corresponds to factor of the form $t_a^{\pm 2}$, or $t_at_b^{-1}$ for a bounding pair $a,b$, in the monodromy $f$. Using Construction \ref{spin_constr1} to find the characteristic sublink for our chosen $\eta \in \Spin(\Sigma_g)$, we set $w_j=0$ if $L_j$ is a component from the link $L$ representing $S^1 \times \Sigma_g$. Using definition \ref{ffu}, and that $L_i$ and $L_{i+1}$ have the same linking number with other components up to sign, we get
\begin{align*}
    (m+1)w_i + mw_{i+1} \equiv mw_i + (m+1)w_{i+1} \\
    \equiv m+1 - \sum_{j \neq i,i+1}lk(L_i, L_j)w_j \pmod{2},
\end{align*}
where $m = \lambda([a],[a])$ and $\lambda$ is the Seifert pairing for the punctured fiber surface in $S^3$. This implies that $w_i \equiv w_{i+1} \equiv m+1 - \sum_{j \neq i,i+1}lk(L_i,L_j)w_j \pmod{2}$. The other components of $\widetilde{L}_f$ come in pairs corresponding to factors of the form $t_c^{\pm 2}$, $t_{d_1}t_{d_2}^{-1}$ appearing in $f$. Since any such pair is either both in, or both out of the characteristic sublink by above, we get that $\sum_{j \neq i,i+1}lk(L_i,L_j)w_j = 0 \pmod{2}$, hence $w_i \equiv w_{i+1} \equiv m+1$. So $L_i$ and $L_{i+1}$ are in the characteristic sublink if and only if $m$ is even.
\end{proof}

There exists a cobordism $Y_f$ between $M_{\widetilde{L}_f}$ and $M_{L_f}$, given in the following way: take $M_{\widetilde{L}_f} \times D^1$ and attach $2$--handles to $M_{\widetilde{L}_f} \times \{1\}$ along the framed link $L$ to obtain $Y_f$. Then $\partial Y_f = M_{\widetilde{L}_f} \sqcup M_{L_f}$. So the boundary of $Y_f$ is the union of $M(\iota,f)$ and the mapping torus $M_f$. Let $W_{L_f}$ be the $2$--handlebody specified by $L_f$, with boundary $M_{L_f}$, and define $W_{\widetilde{L}_f}$ similarly. Then $W_{L_f}$ is the union of $W_{\widetilde{L}_f}$ and $Y_f$ along $M_{\widetilde{L}_f}$, since we can identify $W_{\widetilde{L}_f} \cup_{M_{\widetilde{L}_f}} (M_{\widetilde{L}_f} \times D^1)$ with $W_{\widetilde{L}_f}$ by thinking of $M_{\widetilde{L}_f} \times D^1$ as a collar of the boundary. 

Recall that for the framed link $L_f$, the term $\Lambda_{L_f}$ in the formula of Theorem \ref{kirby_melvin_rochlin_formula} is also the signature of the intersection form of the $2$--handlebody $W_{L_f}$ specified by $L_f$. Then, using Novikov additivity and Lemma \ref{bc_extension_char_sublink}, we have
\begin{align*}
    R(M_{L_f}, \theta_{L_f}(\eta)) = \mathrm{Sign}(W_{L_f}) - \theta_{L_f}(\eta) \cdot \theta_{L_f}(\eta) + 8\Arf(\theta_{L_f}(\eta)) \\
    = \mathrm{Sign}(Y_f) + \mathrm{Sign}(W_{\widetilde{L}_f}) - \theta_{L_f}(\eta) \cdot \theta_{L_f}(\eta) + 8\Arf(\theta_{L_f}(\eta)) \\
    = \mathrm{Sign}(Y_f) + R(M(\iota,f)).
\end{align*}

In summary, we have shown:
\begin{customcor}{4}
    \label{meyer--birman--craggs}
    Define the map 
    \begin{align*}
    \alpha_{\iota}: \Mod_{g,1}[2] \rightarrow \Z/16 \\
    f \mapsto \mathrm{Sign}(Y_f),
    \end{align*}
    where $Y_f$ is the cobordism described above. Then we have $R_{\eta}(f) = \alpha_{\iota}(f) + \mu_{\iota}(f)$, where $\mu_{\iota}:\Mod_{g,1}[2] \rightarrow \Z/16$ denotes the extension of the Birman--Craggs map for the standard embedding $\iota: \Sigma_g \hookrightarrow S^3$. Consequently, $\alpha_{\iota}$ is well--defined, and $\tau_g \equiv \partial(\alpha_{\iota} + \mu_{\iota}) \pmod{16}$.
\end{customcor} 
We finish with a formula for evaluating $\mu_{\iota}$ on an element of $\Mod_{g,1}[2]$. We use the following construction to evaluate $\mu_{\iota}$:
\begin{constr} \label{bc_extension_constr}
Let $[f] = t_{c_n}^{\epsilon_n} \cdots t_{c_1}^{\epsilon_1} \in \Mod_{g,1}[2]$, where $\epsilon_i = \pm 2$ for all $i$. We evaluate $R(M(\iota, f))$ in the following way:
    \begin{enumerate}
        \item Start with the standard Heegaard embedding $\Sigma_g = \partial(\nu ( \Delta_g))$ as in Section \ref{surgery_diagram_construction_section}. The unit normal to $\Sigma_g$ in $S^3$ that points out of the page, defines an embedding of $I \times \Sigma_g$ in $S^3$. Pick $t_1<\cdots < t_n \in I$ and place two parallel copies of $c_i$ in $\{ t_i \} \times \Sigma_g$, both with framing $\lambda([c_i],[c_i]) + \epsilon_i/2$ to obtain $\widetilde{L}_f$. Here, $\lambda$ denotes the Seifert pairing for $\Sigma_g$.
        \item Use Lemma \ref{bc_extension_char_sublink} to find the unique characteristic sublink $C$ of $\widetilde{L}_f$. Then orient each pair of components corresponding to $t_{c_i}^{\epsilon_i}$ oppositely on the Heegaard surface.
        \item Evaluate $R(\widetilde{L}_f, C)$ using Theorem \ref{kirby_melvin_rochlin_formula}.
    \end{enumerate}
\end{constr}
\begin{theorem} \label{bc_extension_formula}
    Let $[f] = t_{c_n}^{\epsilon_n} \cdots t_{c_1}^{\epsilon_1} \in \Mod_{g,1}[2]$, where each $c_i$ is a simple closed curve, and $\epsilon_i = \pm 2$ for all $i=1,..,n$. Then $$\mu_{\iota}([f]) = \Lambda_{\widetilde{L}_f} - \sum_{\lambda([c_i],[c_i]) \equiv 0 \pmod{2}} \epsilon_i \pmod{16},$$
    where $\widetilde{L}_f$ is obtained from Construction \ref{bc_extension_constr}, and $\lambda:H_1(\Sigma_g;\Z) \times H_1(\Sigma_g;\Z) \rightarrow \Z$ denotes the Seifert pairing for the standard inclusion $\iota: \Sigma_g \hookrightarrow S^3$.
\end{theorem}
\begin{proof}
Using Lemma \ref{bc_extension_char_sublink}, we find the unique characteristic sublink of $\widetilde{L}_f$ to be the union of each pair of components corresponding to $t_{c_i}^{\epsilon_i}$ with $\lambda([c_i],[c_i])$ even. For a pair of components $L_i$ and $L_{i+1}$ of this form, note that $L_i$ is isotopic to $L_{i+1}$ ambiently in the Heegaard surface in $S^3$. Since $L_i$ and $L_{i+1}$ are oriented oppositely, a similar argument to the proof of Lemma \ref{satoishom} implies that $L_i$ and $L_{i+1}$ contribute $\epsilon_i$ to the $C \cdot C$ terms in Theorem \ref{kirby_melvin_rochlin_formula}, and that $\Arf(C)=0$ always.
\end{proof}

\printbibliography[heading=subbibliography]

@article{johnsonspin,
	title = {Spin Structures and Quadratic forms on Surfaces},
	journal = {J. London Math. Soc.},
	author = {Johnson, D.},
	date = {1980},
	langid = {english},
}

@article{sato,
	title = {The abelianization of the level d mapping class group},
	journal = {Journal of Topology, Volume 3, Issue 4},
	pages = {847–882},
	author = {Sato, M.},
	date = {2010},
	langid = {english},
}

@book{gompfstipsicz,
    title = {4-manifolds and Kirby calculus},
    author = "Gompf, R. and Stipsicz, A.",
    date = {1999},
    publisher = "A.M.S. Graduate Studies in Mathematics, Volume: 20",
    langid = {english},
}

@article{kirbymelvin1,
	title = {The 3-manifold invariants of Witten and Reshetikhin- Turaev for sl(2, C)},
	journal = {Invent. Math. 105},
	pages = {473-545},
	author = "Kirby, R. and Melvin, P.",
	date = {1991},
	langid = {english},
}

@article{humphriesfactorisation,
	title = {Normal closures of powers of Dehn twists in mapping class groups.},
	journal = {Glasgow Math. J. 34},
	pages = {313-317},
	author = "Humphries, S. P.",
	date = {1992},
	langid = {english},
}

@article{robertelloarf,
	title = {An Invariant of Knot Cobordism},
	journal = {Communications on Pure and Applied Mathematics, Vol XVIII},
	pages = {543-555},
	author = "Robertello, R. A.",
	date = {1965},
	langid = {english},
}

@article{hostearf,
	title = {The Arf Invariant of a Totally Proper Link},
	journal = {Topology and its Applications 18},
	pages = {163-177},
	author = "Hoste, J.",
	date = {1984},
	langid = {english},
}

@book{Hatcher,
    title = {Algebraic Topology},
    author = "Hatcher, A.",
    date = {2002},
    publisher = "Cambridge University Press",
    langid = {english},
}

@article{BCJpaper,
	title = {Quadratic forms and the Birman-Craggs homomorphisms},
	journal = {Transactions of the American Mathematical Society, Vol. 261},
	pages = {235-254},
	author = "Johnson, D.",
	date = {1980},
	langid = {english},
}

@book{kauffmanknots,
    title = {On Knots},
    author = "Kauffman, L.",
    date = {1987},
    publisher = "Princeton University Press AM-115",
    langid = {english},
}

@article{ReshTur2,
	title = {Invariants of 3-manifolds via link polynomials and quantum groups},
	journal = {Inventiones mathematicae},
	pages = {547-597},
	author = "Reshetikhin, N. and Turaev, V.G.",
	date = {1991},
	langid = {english},
}

@article{lickorishannals,
	title = {A Representation of Orientable Combinatorial 3-Manifolds},
	journal = {Annals of Mathematics, Second Series, Vol. 76, No. 3},
	pages = {531-540},
	author = "Lickorish, W.B.R",
	date = {1962},
	langid = {english},
}

@book{farbproblemlist,
    title = {Problems on Mapping Class Groups and Related Topics},
    author = "Farb, B.",
    date = {2006},
    publisher = "Proceedings of Symposia in Pure Mathematics
Volume: 74",
    langid = {english},
}

@article{kmtorusbundle,
	title = {Dedekind sums, mu-invariants and the signature cocycle},
	journal = {Math. Ann. 299},
	pages = {231-267},
	author = "Kirby, R. and Melvin, P.",
	date = {1994},
	langid = {english},
}

@article{birmancraggs,
	title = {The mu-invariant of $3-$manifolds and certain structural properties of the group of homemorphisms of a closed, oriented $2-$manifold},
	journal = {Transactions of the American Mathematical Society, Volume 237},
	pages = {283-309},
	author = "Birman, J. S. and Craggs, R.",
	date = {1978},
	langid = {english},
}

@article{akbulutdotcirc,
	title = {On 2-dimensional homology classes of 4-manifolds},
	journal = {Math. Proc. Camb. Phil. Soc., 82},
	pages = {99-106},
	author = "Akbulut, S.",
	date = {1977},
	langid = {english},
}

@article{johnsonab,
	title = {The structure of the Torelli group-III: The abelianization of I},
	journal = {Topology, Vol. 24, No. 2},
	pages = {127-144},
	author = "Johnson, D.",
	date = {1985},
	langid = {english},
}

@article{johnsonsurvey,
	title = {A survey of the Torelli group},
	journal = {Contemporary Mathematics, Volume 20},
	pages = {165-179},
	author = "Johnson, D.",
	date = {1983},
	langid = {english},
}

@article{lukirbycalc,
	title = {A simple proof of the fundamental theorem of Kirby calculus on links},
	journal = {Transactions of the American Mathematical Society, Vol. 331, Number 1},
	pages = {143-156},
	author = "Lu, N.",
	date = {1992},
	langid = {english},
}

@article{wright,
	title = {The Reshetikhin-Turaev representation of the mapping class group},
	journal = {Journal of Knot Theory and its Ramifications, Vol. 3 No. 4},
	pages = {547-574},
	author = "Wright, G.",
	date = {1994},
	langid = {english},
}

@article{morita,
	title = {Casson's invariant for homology 3-spheres and characteristic classes of surface bundles I},
	journal = {Topology, Vol. 28 No. 3},
	pages = {305-323},
	author = "Morita, S.",
	date = {1989},
	langid = {english},
}

@article{cole,
	title = {Trisections, intersection forms and the Torelli group},
	journal = {Algebraic and Geometric Topology 20},
	pages = {1015-1040},
	author = "Lambert-Cole, P.",
	date = {2020},
	langid = {english},
}

@article{pitschriba,
	title = {Invariants of Z/p-homology 3-spheres from the abelianization of the level-p mapping class group},
	journal = {Quantum Topol. 15},
        pages = {1-85},
	author = "Pitsch, W. and Riba, R.",
	date = {2024},
	langid = {english},
}

@article{polyak,
	title = {A Geometrical Presentation of the Surface Mapping Class Group and Surgery},
	journal = {Commun. Math. Phys. 160},
	pages = {537-556},
	author = "Matveev, S. and Polyak, M.",
	date = {1994},
	langid = {english},
}

@article{putmanduke,
	title = {The Picard group of the moduli space of curves with level structures},
	journal = {Duke Mathematical Journal, Vol.161, No.4},
	pages = {623-674},
	author = "Putman, A.",
	date = {2012},
	langid = {english},
}

@article{meyer,
	title = {Die Signatur von Flachenbundeln},
	journal = {Math. Ann., Vol. 201},
	pages = {239-264},
	author = "Meyer, W.",
	date = {1973},
	langid = {english},
}

@article{blanchetmasbaumspin,
	title = {Topological Quantum Field Theories for surfaces with spin structure},
	journal = {Duke Math. J., Vol. 82, No. 2},
	author = {Blanchet, C. and Masbaum, G.},
	date = {1996},
        pages = {229-267},
	langid = {english},
}

@article{johnsonbp_genus_1_generates,
	title = {Homeomorphisms of a Surface which act trivially on homology},
	journal = {Proc. Amer. Math. Soc., Vol. 75, No. 1},
	author = {Johnson, D.},
	date = {1979},
	langid = {english},
}
\end{document}